\documentclass[mathpazo]{aamm}
\renewcommand\thisyear {200x}

\renewcommand\thisvolume{xx}
\usepackage{subfigure}
\usepackage{multirow}
\usepackage{appendix}
\usepackage{color}
\usepackage{titlesec}
\usepackage{titletoc}

\newcommand\tbbint{{-\mkern -16mu\int}}

\newcommand\dbbint{{-\mkern -19mu\int}}

\newcommand\bbint{
	{\mathchoice{\dbbint}{\tbbint}{\tbbint}{\tbbint}}
}

\begin{document}


\markboth{Juan Zhang, Shifeng Li, Kai Jiang*}{Block preconditioners for mass-conserved Ohta-Kawasaki equation}
\title{Two efficient block preconditioners for the mass-conserved Ohta-Kawasaki equation}

\author[Juan Zhang, Shifeng Li, Kai Jiang*]{Juan Zhang, Shifeng Li, Kai Jiang\corrauth}
\address{School of Mathematics and Computational
	Science, Xiangtan University, Xiangtan, 411105, China,
	Hunan Key Laboratory for Computation and Simulation in Science
	and Engineering, China.}
\email{kaijiang@xtu.edu.cn (Kai Jiang); zhangjuan@xtu.edu.cn (Juan Zhang); shifengli@smail.xtu.edu.cn (Shifeng Li)}

%
%
%

\begin{abstract}
In this paper, we propose two efficient block preconditioners to solve the
mass-conserved Ohta-Kawasaki equation with finite element discretization.
We also study the spectral distribution of these two preconditioners,
\textit{i.e.,} Schur complement preconditioner and the modified Hermitian and
skew-Hermitian splitting (MHSS in short) preconditioner.
Besides, Newton method and Picard method are used to address the implicitly nonlinear term. We rigorously analyze the convergence of Newton method.
Finally, we offer numerical examples to support the theoretical analysis and indicate the efficiency of the proposed preconditioners for the mass-conserved Ohta-Kawasaki equation.
\end{abstract}



\maketitle

\section{Introduction}
\label{sec:intrd}

Diblock copolymers are macromolecules composed of two
incompatible blocks linked together by covalent bonds.
The incompatibility between the two blocks drives the system to phase
separation, while the chemical bonding of the two blocks prevents
the macroscopic phase separation.
These competition factors lead diblock copolymers
to self-assemble into a rich class of complex nanoscale
structures~\cite{bates2012multiblock,jiang2013discovery}.
Modeling and numerical simulation are effective means to
investigate phase behaviors of block copolymers,
such as the self-consistent field theory, and
coarse-grained density functional theory~\cite{leibler1980theory,G.2016,G.2018,fredrickson2006equilibrium}.
Among these theories, Ohta and Kawasaki \cite{T.K.1986}
presented an effective free energy
functional to study diblock copolymers, which can be rescaled as
\begin{align}
	E(u)=\int_{\Omega}(\frac{\epsilon^2}{2} \vert\nabla u\vert^2+\frac{1}{4}(1-u^2)^2+\frac{\sigma}{2}
	\vert(-\Delta)^{-\frac{1}{2}}(u-m)\vert^2)\,dx.
	\label{eq:okenergy}
\end{align}
$u(x)$ is the order parameter that measures the
order of the diblock copolymer system.
$m=\frac{1}{\vert\Omega\vert}\int_{\Omega} u(x)\,dx:=\bbint_{\Omega} u(x)\,dx$ denotes the average mass of the melt on the domain $\Omega$.
The parameters $\epsilon\ll 1$ and $\sigma$ measure
the interfacial thickness in the region of pure phases and the
non-local interaction potential, respectively.
In the energy functional~\eqref{eq:okenergy}, the first term penalizes the jump
in the solution, the second term favors $u=\pm 1$, and the last term
penalizes variation from the mean by a long-range interaction.
More physical background about the Ohta-Kawasaki free energy
functional is described in \cite{T.K.1986}, and corresponding
mathematical theories can be found in the literature
\cite{R.X.2002} and references therein.

Using the Ohta-Kawasaki free energy functional, a mass-conserved dynamic equation can be given as
\begin{align}
	u_t = \Delta \mu.
	\label{eq:dynamic}
\end{align}
$\mu$ is the chemical potential, i.e., the variation derivative of $E$
with respect to $u$
\begin{align}
	\mu = \frac{\delta E}{\delta u} = -\epsilon^2\Delta u -
	u(1-u^2) - \sigma \Delta^{-1} (u-m).
	\label{}
\end{align}
By introducing a new variable $$w=-\epsilon^2\Delta u-u(1-u^2),$$
the fourth-order dynamic equation~\eqref{eq:dynamic} can be split
into two second-order equations on $\Omega \times [0,T]$
\begin{subequations}
	\label{eq:CH}
	\begin{align}
		u_t-\Delta w + \sigma(u-m)=0,\\
		w+\epsilon^2\Delta u-u(u^2-1)=0.
	\end{align}
\end{subequations}
It is easy to verify that the energy functional \eqref{eq:okenergy} is
nonincreasing in time along the solution trajectories of
\eqref{eq:CH} with homogeneous Neumann boundary condition
\begin{align}
	\nabla u \cdot n =0 ~~\mbox{and}~~ \nabla w \cdot n=0 ~~\mbox{on}~~ [0,T]
\end{align}
and the initial value $u(x,0)=u_0(x)$, $x\in \Omega$.

From the numerical computation viewpoint, it is necessary to
construct an efficient numerical method to solve the gradient flow \eqref{eq:CH}.
For time discretization direction, in recent literature, many energy stable approaches
have been proposed, such as convex-splitting schemes \cite{D.1998}, stabilized factor
methods \cite{xu2006stability}, auxiliary variable approaches
\cite{shen2018sav}, and exponential time differencing schemes
\cite{du2004ETD}.
Among these time discretization approaches, the convex splitting scheme (CSS in
short) splits the non-convex nonlinear term into two convex parts and
explicitly-implicitly treats them, which does not have time step restriction in theory. For spatial discretization direction, we adopt the widely used finite element method (FEM in short) in \eqref{eq:CH}.

Solving the discretization system faces twofold difficulties.
One is the implicitly nonlinear term that will take the most time-consuming part in
each time update.
The common approaches to updating the nonlinear term
include Picard method \cite{H.1979,W.2015}, Newton method \cite{J.2017,D.2008},
the nonlinear multigrid method \cite{W.2010,W.2009}, the preconditioned steepest descent algorithm \cite{X.Y.2019,W.2016}, etc.
Picard method and Newton method have been applied to the Ohta-Kawasaki equation \cite{J.2017,P.J.2017}.
However, the convergence rate of Picard method is usually unsatisfactory \cite{X.2019}.
We mainly choose Newton method to update the nonlinear term of the Ohta-Kawasaki equation \eqref{eq:CH}. More significantly, we present a rigorous theoretical analysis of Newton method.

The other difficulty is solving the ill-conditioned linear system at each time step, leading to high computational costs. Due to this, we begin to explore feasible calculation schemes and present corresponding theoretical analyses.
As we know, the previous methods, such as Krylov subspace methods for solving the linear system, may not be convergent or converge slowly without appropriate preconditioners. Recent researchers have paid attention to the construction of preconditioners for the Ohta-Kawasaki equation \cite{J.M.P.2014,L.2018}. Therefore, our other goal is to design efficient preconditioners using Schur complement approximation and MHSS, respectively.

The main contribution of our paper is summarized as follows

\begin{enumerate}
	\item We propose two block preconditioners to solve the ill-conditioned system.
	The Schur complement preconditioner is given by
	\begin{align*}
		\mathcal{P}_{Schur}= \left(
		\begin{array}{cc}
			M &  \zeta S  \\
			-\epsilon^2 S & M+2 \epsilon \sqrt{\zeta} S
		\end{array}
		\right),
	\end{align*}
	and MHSS preconditioner is given by
	\begin{align*}
		\mathcal{P}_{MHSS}=\left(
		\begin{array}{cc}
			A& -\frac{1}{\alpha}AM \\
			M & \zeta S
		\end{array}
		\right),
	\end{align*}
	where $\zeta,~\epsilon,~\alpha$ are given parameters,
	the matrices $M,~S,~A$ are dependent on the numerical discretization scheme of
	the Ohta-Kawasaki equation, see Section \ref{subsec:implementation} for details.
	\item We present an unconditionally stable energy method to \eqref{eq:okenergy} based on the convex-splitting scheme;
	\item We apply Newton method and Picard method to update implicitly nonlinear term, respectively. Meanwhile, the convergence analysis of Newton method is proved.
\end{enumerate}

The rest of the paper is organized as follows.
In Section \ref{sec:method}, Ohta-Kawasaki dynamic equation is discretized
via CSS and FEM. Section \ref{sec:CA} presents the convergence analysis for Newton method.
Section \ref{subsec:implementation} gives two block preconditioners
and analyzes corresponding spectral distributions.
Section \ref{sec:rslts} offers numerical examples to verify theoretical results and demonstrate the efficiency of our proposed method.
Some concluding remarks are drawn in Section \ref{sec:conclusion}.

Throughout the paper, the set of $n \times n$ complex and real
matrices are denoted by $\mathbb{C}^{n\times n}$ and
$\mathbb{R}^{n\times n}$. If $X,~Y\in
\mathbb{R}^{n\times n}$, let $X^{-1}, ~X^{T}, ~\Vert X\Vert_2$ represent
the inverse, conjugate transpose, and the spectral norm of $X$,
respectively.
The expression $X \succ 0$ ($X \succeq 0$) means that $X$
is symmetric (semi-) positive definite. $X \succ Y$ ($X \succeq Y$)
represents that $X-Y$ is symmetric (semi-) positive definite. 
We employ the multi-index notation $\hat\alpha=(\hat\alpha_1,\ldots,\hat\alpha_d)\in \mathbb{Z}^d$
and $\vert\hat\alpha\vert=\vert \hat\alpha_1\vert +\ldots+\vert \hat\alpha_d\vert $.
Let $W^{\hat{m},p}(\Omega)$ be a standard Sobolev space
equipped with the norm $\|\cdot\|_{\hat{m},p}$
and the semi-norm $\vert \cdot\vert_{\hat{m},p}$,
\begin{align*}
	\Vert v\Vert_{\hat{m},p}=
	\begin{cases}
		&\Big{[}\sum\limits_{\vert\hat{\alpha}\vert\leq \hat{m}}\Vert D^{\hat{\alpha}}v\Vert^p_{L^p(\Omega)}
		\Big{]}^{\frac{1}{p}},~~1\leq p<\infty,\\
		&\sum\limits_{\vert\hat{\alpha}\vert\leq
			\hat{m}}\sup\limits_{\Omega} \vert D^{\hat{\alpha}}v\vert,\quad~~p=\infty,
	\end{cases}\\
	\vert v\vert_{\hat{m},p}=
	\begin{cases}		
		&\Big{[}\sum\limits_{\vert\hat{\alpha}\vert
			=\hat{m}}\Vert D^{\hat{\alpha}}v\Vert^p_{L^p(\Omega)}\Big{]}^{\frac{1}{p}},
		~~1\leq p<\infty,\\		
		&\sum\limits_{\vert\hat{\alpha}\vert=
			\hat{m}}\sup\limits_{\Omega} \vert D^{\hat{\alpha}}v\vert,\quad ~~p=\infty.
	\end{cases}
\end{align*}
The notation $\|\cdotp\|$ stands for the $L^2(\Omega)$-norm.

\section{Numerical discretization}
\label{sec:method}

Before we go further, it is necessary to give the weak form
of \eqref{eq:CH}. Using $L^2(\Omega)$-inner product and test function $v\in
H^{1}(\Omega)$, we can have the variational formulation of
\eqref{eq:CH}. Denote the bulk energy density as
\begin{equation}
	\Phi(u)=\frac{1}{4}(1-u^2)^2.
\end{equation}
We seek to find $u(\cdot,t)\in H^{1}(\Omega)$ and $w(\cdot,t)\in
H^{1}(\Omega)$ such that
\begin{equation}
	\begin{split}
		(u_t,v)+(\nabla w, \nabla v) + \sigma(u-m,v)=0,\\
		(w,v)-\epsilon^2(\nabla u, \nabla v)-(\Phi'(u),v)=0,
	\end{split}\label{eq:weakCH}
\end{equation}
where
$$\Phi'(u)=-u(1-u^2).$$
In the following subsections, we will discretize \eqref{eq:weakCH} through CSS and FEM.

\subsection{CSS discretization}
\label{subsec:css}

The CSS, originally proposed by Eyre \cite{D.1998}, splits the non-convex bulk
energy density $\Phi(u)$ into two convex functions:
\begin{align*}
	\Phi(u)=\Phi_{+}(u)-\Phi_{-}(u),
\end{align*}
with
$$\Phi_{+}(u)=\dfrac{1}{4}(u^4+1), ~~\Phi_{-}(u)=\dfrac{1}{2}u^2.$$
Using the above notations, \eqref{eq:weakCH} becomes
\begin{equation}
	\begin{split}
		&(u_t,v)+(\nabla w, \nabla v) + \sigma(u-m,v)=0,\\
		&(w,v)-\epsilon^2(\nabla u,\nabla
		v)-(\Phi'_{+}(u)-\Phi'_{-}(u),v)=0,
	\end{split}\label{eq:css}
\end{equation}
with $$\Phi'_{+}(u)=u^3, ~~\Phi'_{-}(u)=u.$$
The CSS treats $\Phi'_+$ implicitly and $\Phi'_-$ explicitly. In particular, let $\Delta t=T/N$, $N\geq 1$,
be the uniform time step size.
$u^n$ and $w^n$ represent the approximation of $u(\cdot, t_n)$
and $w(\cdot, t_n )$ in $H^{1}(\Omega)$, where $t_n=n\Delta t$, $n=0, 1,\cdots,N$.
The continuous system \eqref{eq:weakCH} can be discretized as
\begin{equation}
	\begin{split}
		\Big(\frac{u^n-u^{n-1}}{\Delta t},v \Big)+(\nabla w^n, \nabla v) + \sigma(u^n - m ,v)=0,\\
		(w^n,v)-\epsilon^2(\nabla u^n, \nabla
		v)-(\Phi'_+(u^n)-\Phi'_-(u^{n-1}), v)=0,
	\end{split}\label{eq:semidiscretedCH}
\end{equation}
where the test function $v\in H^{1}(\Omega)$.
The CSS \eqref{eq:semidiscretedCH} has the energy dissipation law as shown in
Theorem \ref{thm:dissp}.
\begin{theorem}\label{thm:dissp}
	The scheme \eqref{eq:semidiscretedCH} is unconditionally energy stable. Moreover, we have
	\begin{align}
		E(u^{n})\leq E(u^{n-1})- \Delta t\|\nabla \mu\|^2 -
		\frac{\varepsilon^2}{2}\|\nabla (u^n-u^{n-1})\|^2.\label{eq:energydissp}
	\end{align}
\end{theorem}

\begin{proof}
	See Appendix A.
\end{proof}

\subsection{Finite element discretization}
\label{subsec:fem}

Next, we discretize the semi-discrete system
\eqref{eq:semidiscretedCH} in space by FEM.
Denote $X=H^{1}(\Omega)=W^{1,2}(\Omega)$.
$h$ is a real positive number and
$\tau_h=\{K_{\Omega}: \cup_{K_{\Omega}\subset\Omega}\bar{K}_{\Omega}=\bar{\Omega}\}$ is a quasi-uniform regular
partition of $\Omega$ with diameters bounded by $h$.
For a given $\tau_h$, FEM space $V_h\subset X$ is defined by
\begin{align*}
	V_h=\{v_h \in C(\bar{\Omega})\cap X : v_h\vert_{K_{\Omega}} \in P_1(K_{\Omega}),
	~~\forall\, K_{\Omega}\in \tau_h\},
\end{align*}
where $P_1(K_{\Omega})$ is the polynomial space of degree, not greater than $1$.
The full-discrete system for \eqref{eq:weakCH} is
obtained by seeking for the test function $v_h \in V_h$ such that
\begin{subequations}
	\label{eq:fullCH}
	\begin{align}
		\label{eq:fullCH-1}
		\Big(\frac{u^n_h-u^{n-1}_h}{\Delta t},v_h \Big)+(\nabla w^n_h, \nabla v_h)
		+ \sigma(u^n_h - m ,v_h )=0,\\
		\label{eq:fullCH-2}
		(w^n_h,v_h)-\epsilon^2(\nabla u^n_h, \nabla
		v_h)-(\Phi'_+(u^n_h)-\Phi'_-(u^{n-1}_h), v_h)=0.
	\end{align}
\end{subequations}

The sequence $u_h^n$ generated by \eqref{eq:fullCH} is bounded uniformly in $h$, as
shown in Theorem \ref{thm:bound}.
\begin{theorem}\label{thm:bound}
	The sequence $u_h^n$ defined by \eqref{eq:fullCH} is bounded, i.e.,
	\begin{align}
		\|u_h^n\|_{L^\infty(\Omega)}\leq C(\epsilon, \sigma, u_0, m, c, c_s, T,\vert \Omega\vert),
		~~n=1,2,\cdots,N,\label{eq:bound}
	\end{align}
	where $c$ and $c_s$ depend only on the space dimension $d$ and $\Omega$.
\end{theorem}

\begin{proof}
	See Appendix B.
\end{proof}

\section{Convergence analysis}
\label{sec:CA}
This section applies Newton method and Picard method to treat the nonlinear term.
The Newton method of linearizing $(u^n_h)^3$
in \cite{J.2017} and \cite{P.J.2017} is
\begin{align}
	(u^n_h)^3 \approx (u^{n,k-1}_h)^3 + 3 (u^{n,k-1}_h)^2(u^{n,k}_h- u^{n,k-1}_h).
	\label{eq:newton}
\end{align}
The Picard method of linearizing $(u^n_h)^3$
in \cite{L.2018} and \cite{P.Q.2012} is
\begin{align}
	(u^n_h)^3 \approx (u^{n,k-1}_h)^2\,u^{n,k}_h,
	\label{eq:vn}
\end{align}
where $u^{n,0}_h=u^{n-1}_h.$

To the best of our knowledge, there has been no convergence analysis for
Newton method in related literature. One of our main works is providing the convergence analysis for Newton method and compares the convergence rates of these two methods in our numerical examples. Next, we establish the theoretical analysis for Newton method.

\subsection{Bilinear forms}
From \eqref{eq:fullCH}, we have
\begin{align}
	\epsilon^2(\nabla u^n_h, \nabla
	v_h)+((u^n_h)^3, v_h)+(1+\Delta t\sigma)(u^n_h,v_h)
	-(w^n_h,v_h)+\Delta t(\nabla w^n_h, \nabla v_h)
	=(f,v_h),\label{3.2}
\end{align}
where
\begin{align}
	f=2u^{n-1}_h+\Delta t\sigma m\label{definition:f}.
\end{align}
To simplify the following analysis, we introduce the bilinear forms.
For $u,~w\in X$, denote
\begin{align*}
	\Psi(u,w;v)&=\epsilon^2(\nabla u,\nabla v)
	+(1+\Delta t\sigma)(u,v)+(u^3,v)-(f,v)
	+\Delta t(\nabla w,\nabla v)-(w,v)\\
	&=\Psi_1(u,v)+\Psi_2(w,v),
\end{align*}
where
\begin{align*}
	\Psi_1(u,v)&=(F_1(\cdot,u,\nabla u),\nabla v)
	+(g_1(\cdot,u,\nabla u), v),\\
	\Psi_2(w,v)&=(F_2(\cdot,w,\nabla w),\nabla v)
	+(g_2(\cdot,w,\nabla w), v),
\end{align*}
and
\begin{align*}
	&F_1(x,u,\nabla u)=\epsilon^2\nabla u,~g_1(x,u,\nabla u)=(1+\Delta t\sigma) u+u^3-f,\\
	&F_2(x,w,\nabla w)=\Delta t \nabla w,~g_2(x,w,\nabla w)=-w.
\end{align*}
Therefore, \eqref{3.2} is equivalent to
\begin{align}
	\Psi(u^n_h,w^n_h;v_h)=0.\label{3.3}
\end{align}
Evidently, $\Psi(u,w;v)$ is a linear combination of bilinear forms
$\Psi_1$ and $\Psi_2$.
For $i=1,2$, note that
\begin{align*}
	F_i(x,y,z):\Omega\times \mathbb{R}^1 \times \mathbb{R}^d\rightarrow
	\mathbb{R}^d ~~\mbox{and}~~ g_i(x,y,z):\Omega\times \mathbb{R}^1 \times
	\mathbb{R}^d\rightarrow \mathbb{R}^1
\end{align*}
are smooth functions. For any $u, w\in X$, we have
\begin{align*}
	&F_{1,z}(x,u,\nabla u)=\epsilon^2\,\mathbf{I}_2,~F_{1,y}(x,u,\nabla u)=\mathbf{0}_2,\\
	&g_{1,z}(x,u,\nabla u)=\mathbf{0}_2,~g_{1,y}(x,u,\nabla u)=1+\Delta t\sigma+3u^2,\\
	&F_{2,z}(x,w,\nabla w)=\Delta t\,\mathbf{I}_2,~F_{2,y}(x,w,\nabla w)=\mathbf{0}_2,\\
	&g_{2,z}(x,w,\nabla w)=\mathbf{0}_2,~g_{2,y}(x,w,\nabla w)=-1,
\end{align*}
where $\mathbf{I}_2$ and $\mathbf{0}_2$ represent 2-order identity matrix and 2-order zero matrix, respectively.

Similar to \cite{H.2013}, from \eqref{3.3}, we introduce the bilinear form
\begin{align}
	\Psi^{'}(u,w;v_1,v_2,v)=\Psi^{'}_1(u;v_1,v)+\Psi^{'}_2(w;v_2,v),
	\label{le:00}
\end{align}
where
\begin{align*}
	\Psi^{'}_1(u;v_1,v)
	&=(F_{1,z}\nabla v_1+F_{1,y}v_1,\nabla v)+(g_{1,z}\,\nabla v_1+g_{1,y}v_1,v)\\
	&=(\epsilon^2\nabla v_1,\nabla v)
	+(1+\Delta t \sigma)(v_1,v)+3(u^2v_1,v),\\
	\Psi^{'}_2(w;v_2,v)&=(F_{2,z}\nabla v_2
	+F_{2,y}v_2,\nabla v)+(g_{2,z}\,\nabla v_2+g_{2,y}v_2,v)\\
	&=\Delta t(\nabla v_2,\nabla v)-(v_2,v).
\end{align*}
Obviously,
\begin{align}
	\Psi_2(w,v)=\Psi^{'}_2(\cdot;w,v).\label{3.4}
\end{align}

\subsection{Convergence analysis for Newton method}

Next, we present the error estimate for Newton method.
Applying \eqref{eq:newton} to \eqref{eq:semidiscretedCH} yields 
\begin{subequations}
	\label{eq:Newtonnonlinear}
	\begin{align}
		&(1+\sigma \Delta t)(u^{n,k}_h, v_h)
		+\Delta t(\nabla w^{n,k}_h,\nabla v_h)
		=\Delta t \sigma (m,v_h)+(u^{n-1}_h,v_h),\\
		&(w^{n,k}_h,v_h)-\epsilon^2(\nabla u^{n,k}_h,\nabla v_h)
		+(-3(u^{n,k-1})^2(u^{n,k}_h),v_h)
		+(u^{n-1}_h,v_h)=-2((u^{n,k-1})^3,v_h).
	\end{align}
\end{subequations}
\eqref{eq:Newtonnonlinear} can be rewritten as
\begin{align*}
	&\Psi^{'}_1(u_h^{n,k-1};u_h^{n,k},v_h)
	+\Psi^{'}_2(w_h^{n,k-1};w_h^{n,k},v_h)= \\
	&\Psi^{'}_1(u_h^{n,k-1};u_h^{n,k-1},v_h)
	+\Psi^{'}_2(w_h^{n,k-1};w_h^{n,k-1},v_h)\\
	&-\Psi_1(u^{n,k-1},v_h)-\Psi_2(w^{n,k-1},v_h),~\forall v_h\in V_h.
\end{align*}
Therefore,
\begin{align}
	\Psi^{'}_1(u_h^{n,k-1};u_h^{n,k},v_h)
	+\Psi^{'}_2(w_h^{n,k-1};w_h^{n,k},v_h)=
	\Psi^{'}_1(u_h^{n,k-1};u_h^{n,k-1},v_h)
	-\Psi_1(u^{n,k-1},v_h).\label{Newton:2}
\end{align}
The error estimate for Newton method is illustrated by Theorem \ref{th:Newton Error}.

\begin{theorem}
	\label{th:Newton Error}
	Let $u_h^{n}$ and $u_h^{n,\tilde{m}}$ be the solutions of problems \eqref{eq:fullCH} and \eqref{eq:Newtonnonlinear}. Assume that $h$ is sufficiently small, take $u^{n,0}_h=u^{n-1}_h$ and $0\leq\eta<1$ such that
	\begin{align}
		\Vert u^{n,0}_h-u_h^{n}\Vert_{1,\infty}\leq \eta.\label{initial}
	\end{align}
		For positive constants $c_1,~c_2,~c_3,~C$ and $2\leq p\leq\infty$, if
		the convergence factor $\rho$ satisfies
\begin{align}
\rho=\frac{c_3\eta C\vert\log h\vert}
{1-c_3[(C+1)(c_2+\eta)+c_1]\vert\log h\vert}<1,
\label{3.13}
\end{align}
then
	\begin{align}
		\Vert u^n_h-u_h^{n,\tilde{m}}\Vert_{1,p}\leq c\rho^{2^{\tilde{m}}},
		\label{3.14}
	\end{align}
	where $c$ is a canstant.
\end{theorem}

\begin{proof}
	See Appendix C.
\end{proof}


Besides, we show the fully-discrete scheme for Picard method.
Applying \eqref{eq:vn} to \eqref{eq:semidiscretedCH} yields
\begin{subequations}
	\label{eq:Picardnonlinear}
	\begin{align}
		&(1+\sigma \Delta t)(u^{n,k}_h, v_h)+\Delta t(\nabla w^{n,k}_h,\nabla v_h)=\Delta t \sigma (m,v_h)+(u^{n-1}_h,v_h),\label{eq:Picardnonlinear1}\\
		&(w^{n,k}_h,v_h)-\epsilon^2(\nabla u^{n,k}_h,\nabla v_h)+(-(u^{n,k-1})^2(u^{n,k}_h),v_h)+(u^{n-1}_h,v_h)=0.\label{eq:Picardnonlinear2}
	\end{align}
\end{subequations}
Picard method has linear convergence and the convergence speed of Newton method is faster than that of Picard method. It will be verified by numerical experiments in Section \ref{sec:rslts} as well.

%

\section{Block preconditioners}

\label{subsec:implementation}

In practical implementation, we impose a uniform discretization on
the spatial domain $\tau_h$ for each $t_n=n\Delta t$ $(0\leq n\leq N)$.
Denote the dimension of $V_h$ as $\kappa$. We use the set of piecewise linear $\phi_i$ as the basis functions that are defined in the usual way.
Therefore, $V_h$ can be spanned in terms of $\phi_i$ as
\begin{align*}
	V_h = \mbox{span}\{\phi_i\}_{i=0}^{\kappa-1},
\end{align*}
$u_h^n$ and $w_h^n$ can be expressed as
\begin{align*}
	u_h^n=\sum_{i=0}^{\kappa-1}U_i^n\phi_i(x), ~~~
	w_h^n=\sum_{i=0}^{\kappa-1}W_i^n\phi_i(x).
\end{align*}
Let $v_h=\phi_j$, the discretized system \eqref{eq:fullCH} can be written as
\begin{subequations}
	\label{eq:implCH}
	\begin{align}
		\label{eq:implCH-1}
		(1+\sigma \Delta t )\sum_{i=0}^{\kappa-1}U_i^n(\phi_i,\phi_j)+\Delta t \sum_{i=0}^{\kappa-1}W_i^n(\nabla\phi_i,\nabla\phi_j)
		=\sum_{i=0}^{\kappa-1}U_i^{n-1}(\phi_i,\phi_j)+\sigma\Delta t(m,\phi_j),\\
		\label{eq:implCH-2}
		\sum_{i=0}^{\kappa-1}W_i^n(\phi_i,\phi_j)
		=\epsilon^2
		\sum_{i=0}^{\kappa-1}U_i^n(\nabla\phi_i,\nabla\phi_j)+\Bigg{(}\Big(\sum_{i=0}^{\kappa-1}
		U_i^n \phi_i \Big)^3-\sum_{i=0}^{\kappa-1} U_i^{n-1}
		\phi_i,\phi_j\Bigg{)}.
	\end{align}
\end{subequations}

We use Newton method or Picard method to approximate the implicit nonlinear term in \eqref{eq:implCH-2}.
For simplicity, we show the implementation scheme of Newton method for each $j ~(0\leq j\leq \kappa-1)$ below
\begin{align}
	\Bigg(\Big(\sum_{i=0}^{\kappa-1} U_i^n \phi_i\Big)^3,\phi_j\Bigg)
	&\approx \Bigg{(}3\sum_{i=0}^{\kappa-1}U_i^{n,k} \phi_i \Big(\sum_{l=0}^{\kappa-1} U_l^{n,k-1}
	\phi_l \Big)^2,\phi_j\Bigg{)}
	-\Bigg{(}2\Big(\sum_{l=0}^{\kappa-1} U_l^{n,k-1}\phi_l \Big)^3,\phi_j\Bigg{)}\notag\\
	&=\sum_{i=0}^{\kappa-1} U_i^{n,k} \Big{(} 3\int_\Omega
	(\sum_{l=0}^{\kappa-1} U_l^{n,k-1} \phi_l)^2\phi_i\phi_j \,dx \Big{)}
	-2\int_{\Omega}u^{n,k-1}_h\phi_j\,dx.\label{1.13}
\end{align}

Define the mass and stiffness matrices as
\begin{align*}
	&M=(m_{ij}), ~~m_{ij}=(\phi_i,\phi_j)=\int_{\Omega}\phi_i\phi_j\,dx,\\
	&S=(s_{ij}),
	~~s_{ij}=(\nabla\phi_i,\nabla\phi_j)=\int_{\Omega}\nabla\phi_i
	\cdot \nabla\phi_j \,dx.
\end{align*}
Note that $M\succ 0$ and $S\succeq 0$.
For $0 \leq j \leq \kappa-1$, let
\begin{align*}
	&L^{(n,k)}=(l_{ij}^{(n,k)}),
	~l_{ij}^{(n,k)}
	=3\int_\Omega(\sum_{l=0}^{\kappa-1} U_l^{n,k-1} \phi_l)^2 \phi_i\phi_j\, dx, \\
	& U^{(n-1)} = (U_0^{(n-1)}, \cdots, U_{\kappa-1}^{(n-1)})^T,\\
	&F^{(n)}=(F_{0}^{(n)},\cdots,~F_{\kappa-1}^{(n)})^{T},~ F_{j}^{n}=(MU^{n-1})_{j}+\sigma\Delta t(m,\phi_j),\\
	&E^{(n,k)}=(E_{0}^{(n,k)},\cdots,E_{\kappa-1}^{(n,k)})^{T},~E_{j}^{n,k}=-2\int_{\Omega}u^{n,k-1}_h\phi_j\,dx.
\end{align*}
$L^{(n,k)}\succeq 0$ depends on the previous iteration solution at each time step.
The vectors $F^{n}$ and $E^{n}$ are obtained from the previous step.
Using the approximation \eqref{1.13}, \eqref{eq:implCH} is summarized as follows
\begin{align}
	\mathcal{A}\left(
	\begin{array}{c}
		U^{(n,k)}  \\
		W^{(n,k)}   \\
	\end{array}
	\right)=\left(
	\begin{array}{cc}
		(1+\sigma \Delta t)M &  \Delta t S \\
		-\epsilon^2 S-L^{(n,k)} & M   \\
	\end{array}
	\right) \left(
	\begin{array}{c}
		U^{(n,k)}  \\
		W^{(n,k)}   \\
	\end{array}
	\right)= \left(
	\begin{array}{c}
		F^{(n)}\\
		E^{(n,k)}
	\end{array}
	\right), \label{eq:discretedmatrix}
\end{align}
where
\begin{align*}
	U^{(n+1)} = \lim_{k\rightarrow \infty} U^{(n,k)}, ~~~
	W^{(n+1)} = \lim_{k\rightarrow \infty} W^{(n,k)}.
\end{align*}
The initial conditions are $U^{(n,0)}=U^{(n-1)}$ and
$W^{(n,0)}=W^{(n-1)}$.

Subsequently, it is required to solve the
discretized linear problem \eqref{eq:discretedmatrix} using linear
solvers at each time step. Due to a fast increase in memory
requirement and bad scaling properties for massively parallel
problems, the direct solvers, like UMFPACK \cite{T.A.D.2004},
MUMPS \cite{P.I.J.2001}, or SuplerLU-DIST \cite{X.J.2003},
may be difficult to solve \eqref{eq:discretedmatrix} efficiently.
Hence we use iteration methods to address these problems, such as Krylov subspace methods. The linear system \eqref{eq:discretedmatrix}
becomes more ill-conditioned as the mesh is refined.
For instance, in one dimension, Figure \ref{FIG:4} and Table \ref{tab5}
give the spectral distribution and corresponding condition
number for different subdivisions.
The ill-conditioned system reduces the performance of linear solvers and impedes
the convergence of nonlinear solvers, which is the
difficulty in numerical computation.

To improve the condition number of the linear system and
accelerate the convergence of nonlinear iteration, natural selection is using a preconditioning approach. In general, minimizing the spectral radius of the iteration matrix is not necessarily the best choice, which also holds when the algorithm is used as a preconditioner for a Krylov subspace method. For GMRES, the convergence results do not depend on the spectrum of the coefficient matrix alone \cite{Greenbaum1996nonincreasing, Arioli1998Krylov}. However, numerical calculations usually show that the tight clustering of eigenvalues for the preconditioned system frequently leads to strong convergence properties and makes the algorithm efficient. Therefore, we give the spectral analysis of proposed preconditioners.


\begin{figure}[htb]
	\centering
	\label{fig:a}
	\includegraphics[width=1.8in]{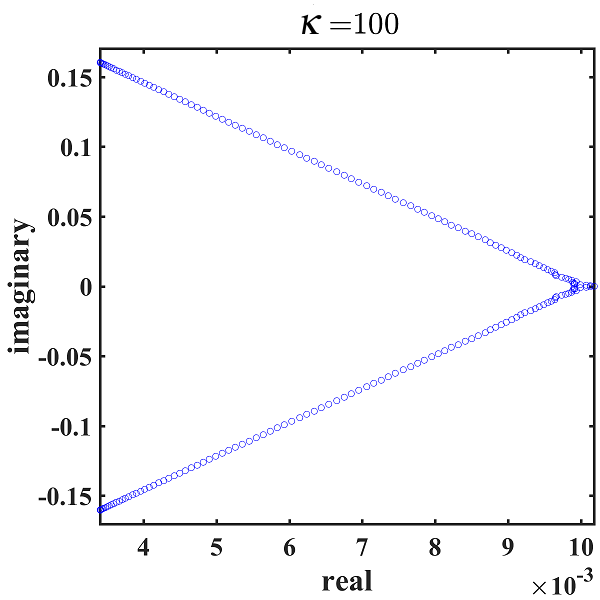}
	\hspace{0.2in}
	\label{fig:b}
	\includegraphics[width=1.8in]{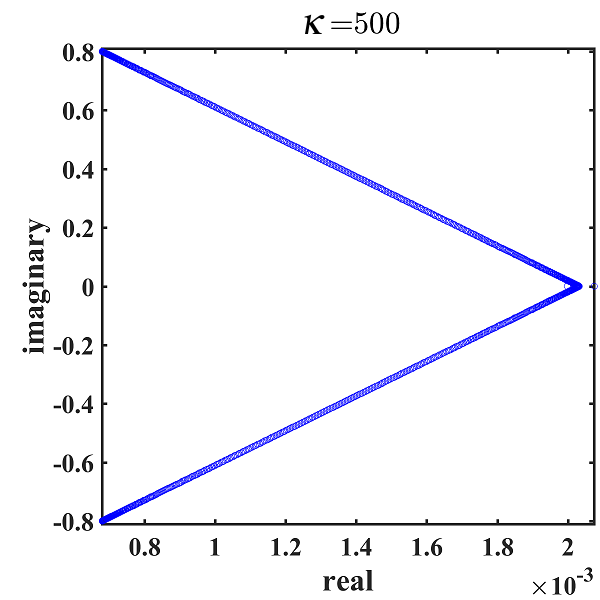}
	\label{fig:c}
	\includegraphics[width=1.8in]{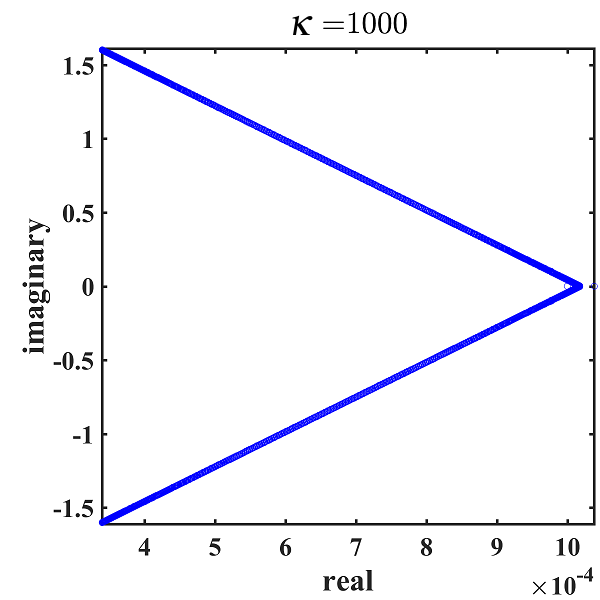}
	\hspace{0.2in}
	\label{fig:d}
	\includegraphics[width=1.8in]{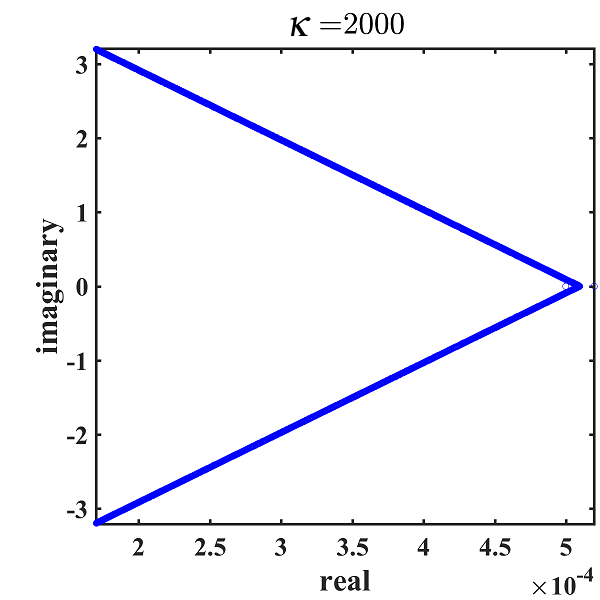}
	\caption{The spectral distribution of the linear system
		\eqref{eq:discretedmatrix} for one dimensional dynamic equation with different subdivision
		$\kappa$ when $\epsilon=0.1$, $\Delta t=0.01$,
		$\sigma=100$, $n=1$ and $k=1$.}
	\label{FIG:4}
\end{figure}
\begin{table}[htb]
	\centering
	\caption{The condition number of the linear system
		\eqref{eq:discretedmatrix} for one dimensional dynamic equation with different subdivision
		$\kappa$ when $\epsilon=0.1$, $\Delta t=0.01$, $\sigma=100$, $n=1$ and $k=1$.}
	\label{tab5}
	\renewcommand\arraystretch{1}
	\setlength{\tabcolsep}{5mm}{
		\begin{tabular}{|c|c|c|c|c|}\hline
			$\kappa$                                &100          &500          & 1000                         & 2000 \\ \hline
			$\mbox{cond}(\mathcal{A})$              &0.407e+03           &1.009e+04          & 4.037e+04           &1.624e+05\\ \hline
	\end{tabular}}\\
\end{table}

\subsection{Schur complement preconditioner}
\label{sec:existpred}

From the above discussion, it is necessary to find appropriate
preconditioners for solving the discretized system \eqref{eq:discretedmatrix}.
In \cite{P.J.2017} and \cite{L.2018}, two preconditioners have been proposed for \eqref{eq:okenergy} when using the backward Euler scheme. In this subsection, we will present Schur complement preconditioners for CSS similar to the construction technique in \cite{L.2018}. Meanwhile, we will analyze the spectral distribution of the preconditioning system.

Denote $\zeta=\Delta t/(1+\sigma \Delta t).$
	Through scaling the first row of $\mathcal{A}$ by $\zeta$, we have
	\begin{align*}
		\hat{\mathcal{A}}=\left(
		\begin{array}{cc}
			M &  \zeta S \\
			-\epsilon^2 S-L^{(n,k)} & M   \\
		\end{array}
		\right).
	\end{align*}
Using Schur complement method and ignoring the influence of $L^{(n,k)}$,
the Schur block preconditioner (Schur in short) for \eqref{eq:discretedmatrix} is
\begin{align*}
	\mathcal{P}_{Schur}= \left(
	\begin{array}{cc}
		M &  \zeta S  \\
		-\epsilon^2 S & M+2 \epsilon \sqrt{\zeta} S
	\end{array}
	\right).
\end{align*}
The spectral distribution for the preconditioning system is shown as follows.
\begin{theorem}\label{BTeig}
	The eigenvalues of $\mathcal{P}_{Schur}^{-1}\hat{\mathcal{A}}$ are real and satisfy
	\begin{align*}
		\lambda(\mathcal{P}_{Schur}^{-1}\hat{\mathcal{A}})\in (\frac{1}{2},1+\frac{1}{4\epsilon}\sqrt{\zeta}\lambda_{+}),
	\end{align*}
	with $\lambda_{+}=\lambda_{\mbox{max}}(M^{-1}L^{(n,k)})$.
\end{theorem}
\begin{proof}
	Note that
	\begin{align*}
		\tilde{K}=(M+\epsilon \sqrt{\zeta} S)M^{-1}(M+\epsilon \sqrt{\zeta} S)
	\end{align*}
	is a good approximation of the Schur complement \cite{Pearson2012new}\cite{Pearson2018matching}
	\begin{align*}
		K=M+\epsilon^2 \zeta SM^{-1}S+\zeta L^{(n,k)}M^{-1}S.
	\end{align*}
	It is obvious that
	\begin{align}
		\mathcal{P}_{Schur}^{-1}\hat{\mathcal{A}}
		=\begin{pmatrix}
			I & \zeta M^{-1} S\\
			0 & I
		\end{pmatrix}^{-1}
		\begin{pmatrix}
			I             &   0\\
			\tilde{K}^{-1}L^{(n,k)}& \tilde{K}^{-1}K
		\end{pmatrix}
		\begin{pmatrix}
			I & \zeta M^{-1} S\\
			0 & I
		\end{pmatrix}.
		\label{0.38}
	\end{align}
	Therefore, it is required to prove
	\begin{align*}
		\lambda(\tilde{K}^{-1}K)\in (\frac{1}{2},1+\frac{1}{4\epsilon}\sqrt{\zeta}\lambda_{+}).
	\end{align*}
	
	Suppose the corresponding eigenvector $v$ satisfies
	\begin{align*}
		Kv=\lambda \tilde{K}v.
	\end{align*}
	
	i) If $v\in \mbox{null}(S)$, then
	$Mv=\lambda Mv$, which implies that $\lambda=1$.
	
	ii) If $v\notin \mbox{null}(S)$, we have
	$$SM^{-1}Kv=\lambda SM^{-1}\tilde{K}v.$$
	It follows that
	\begin{align}
		v^*Fv=\lambda v^*Gv,
		\label{1.01}
	\end{align}
	where
	\begin{align*}
		&F=SM^{-1}K=S+\epsilon^2 \zeta SM^{-1}SM^{-1}S+\zeta SM^{-1}L^{(n,k)}M^{-1}S,\\
		&G=SM^{-1}\tilde{K}=S+\epsilon^2 \zeta SM^{-1}SM^{-1}S+2\epsilon \sqrt{\zeta} SM^{-1}S.
	\end{align*}
	Since $F$ and $G$ are real symmetric, $v^*Fv$ and $v^*Gv$ are real (with $v^*Gv$ also positive).
	The eigenvalues of $\lambda(\tilde{K}^{-1}K)$ are real.
	Using \eqref{1.01}, we get
	\begin{align}
		\nonumber\lambda=\frac{v^*Fv}{v^*Gv}&=\frac{v^*Sv+\epsilon^2 \zeta v^*SM^{-1}SM^{-1}Sv+\zeta v^*SM^{-1}L^{(n,k)}M^{-1}Sv}
		{v^*Sv+\epsilon^2 \zeta v^*SM^{-1}SM^{-1}Sv + 2\epsilon \sqrt{\zeta} v^*SM^{-1}Sv}\\
		\nonumber&=1-\frac{2\epsilon \sqrt{\zeta} v^*SM^{-1}Sv}{v^*Sv+2\epsilon \sqrt{\zeta} v^*SM^{-1}Sv+\epsilon^2 \zeta v^*SM^{-1}SM^{-1}Sv}\\
		\nonumber&~~~~+\frac{\zeta v^*SM^{-1}L^{(n,k)}M^{-1}Sv}{v^*Sv+2\epsilon \sqrt{\zeta} v^*SM^{-1}Sv+\epsilon^2 \zeta v^*SM^{-1}SM^{-1}Sv}\\
		&=1-\mathcal{R}_{1}+\mathcal{R}_{2}.
		\label{1.02}
	\end{align}
	It is evident that $\mathcal{R}_{1}>0$.
	Writing
	$$a=S^{\frac{1}{2}}v,~b=\epsilon \sqrt{\zeta} S^{\frac{1}{2}}M^{-1}Sv,$$
	combining with the inequality $a^*a+b^*b\geq a^*b+b^*a$,
	it follows that
	$$\mathcal{R}_{1}=(a^*b+b^*a)/(a^*a+b^*b+a^*b+b^*a)\leq \frac{1}{2}.$$
	Therefore, we have
	\begin{align}
		\mathcal{R}_{1}\in (0,\frac{1}{2}].
		\label{1.03}
	\end{align}
	Since
	\begin{align*}
		\mathcal{R}_{2}&=\zeta \frac{v^*SM^{-1}L^{(n,k)}M^{-1}Sv}{v^*SM^{-1}Sv}
		\cdotp \frac{v^*SM^{-1}Sv}{v^*Sv+2\epsilon \sqrt{\zeta} v^*SM^{-1}Sv+\epsilon^2 \zeta v^*SM^{-1}SM^{-1}Sv}\\
		&=\zeta \mathcal{R}_{21} \cdotp \mathcal{R}_{22},
	\end{align*}
	then
	\begin{align*}
		\mathcal{R}_{22}=\frac{1}{2\epsilon \sqrt{\zeta}} \mathcal{R}_{1},
	\end{align*}
	which shows that
	\begin{align*}
		\mathcal{R}_{22}\in (0,\frac{1}{4\epsilon \sqrt{\zeta}}].
	\end{align*}
	Denote $\bar{v}=M^{-\frac{1}{2}}Sv$, we immediately have
	\begin{align*}
		\mathcal{R}_{21}=\frac{\bar{v}^*M^{-\frac{1}{2}}L^{(n,k)}M^{-\frac{1}{2}}\bar{v}}{\bar{v}^*\bar{v}}
		\in [\lambda_{\mbox{min}}(M^{-1}L^{(n,k)}), \lambda_{\mbox{max}}(M^{-1}L^{(n,k)})],
	\end{align*}
	where $\lambda_{\mbox{min}}(M^{-1}L^{(n,k)}) > 0$. Combining with the derived bounds, we can obtain
	\begin{align}
		\mathcal{R}_{2}\in (0,\frac{1}{4\epsilon} \sqrt{\zeta}\lambda_+].
		\label{1.04}
	\end{align}
	Substituting \eqref{1.03} and \eqref{1.04} into \eqref{1.02}, it follows that
	\begin{align*}
		\lambda(\tilde{K}^{-1}K)\in (\frac{1}{2},1+\frac{1}{4\epsilon}\sqrt{\zeta}\lambda_{+}),
	\end{align*}
	and this proof is completed.
\end{proof}

\subsection{MHSS preconditioner}
\label{sec:ourmethod}

In this subsection, we propose a block preconditioner based on MHSS and analyze the eigenvalue distribution of the preconditioning system.

The linear system \eqref{eq:discretedmatrix} can be rearranged equivalently to the two-by-two block system
\begin{align}
	\tilde{\mathcal{A}}\left(
	\begin{array}{c}
		U^{(n,k)}  \\
		W^{(n,k)}   \\
	\end{array}
	\right)=\left(
	\begin{array}{cc}
		\epsilon^2 S+L^{(n,k)}& -M \\
		M & \zeta S\\
	\end{array}
	\right) \left(
	\begin{array}{c}
		U^{(n,k)}  \\
		W^{(n,k)}   \\
	\end{array}
	\right)=
	\left(
	\begin{array}{c}
		-E^{(n,k)}\\
		\frac{1}{1+\sigma \Delta t}F^{(n)}
	\end{array}
	\right)=b.\label{1.15}
\end{align}
Let
$A=\epsilon^2 S+L^{(n,k)}, ~B=\zeta S$, then
\begin{align*}
	\tilde{\mathcal{A}}=\left(
	\begin{array}{cc}
		A& -M \\
		M & B\\
	\end{array}
	\right),
\end{align*}
where $A\in \mathbb{R}^{\kappa \times \kappa }$ and
$A\succ 0$, $B\in \mathbb{R}^{\kappa\times \kappa}$ and $B\succeq 0$.
The rearranged linear system \eqref{1.15} is a
generalized saddle point problem, which widely exists in
scientific computing and numerical algebra.

In recent years, many works have been devoted to developing
efficient preconditioners for the generalized saddle point problem, such as block diagonal and triangular preconditioners
(\cite{O.2015}-\cite{Y.C.J.2010}), matrix splitting
preconditioners (\cite{G.Z.2016}-\cite{L.L.G.2019}).
Among the preconditioners, HSS preconditioner is an efficient
method to solve the generalized saddle point problem, originally
developed by Bai et al.~\cite{G.B.2003}.
As an improvement, we present an MHSS block preconditioner for the linear system \eqref{1.15}.
Using the similar constructing technique as \cite{L.L.G.2019}, for any
constant $\alpha>0$, MHSS block preconditioner
$\mathcal{P}_{MHSS}$ is defined by
\begin{align}
	\mathcal{P}_{MHSS}=\left(
	\begin{array}{cc}
		\dfrac{1}{\alpha}A& 0 \\
		0 &I\\
	\end{array}
	\right)\left(
	\begin{array}{cc}
		\alpha I& -M \\
		M & B
	\end{array}
	\right).\label{1.16}
\end{align}
Let
\begin{align*}
	\mathcal{R}=\mathcal{P}_{MHSS}-\tilde{\mathcal{A}}=\left(
	\begin{array}{cc}
		0& M-\dfrac{1}{\alpha}AM \\
		0 & 0\\
	\end{array}
	\right),
\end{align*}
then
\begin{align*}
	\tilde{\mathcal{A}}=\mathcal{P}_{MHSS}-\mathcal{R}.
\end{align*}
	MHSS preconditioner is much closer to the coefficient matrix $\tilde{\mathcal{A}}$ than HSS preconditioner and outperforms HSS preconditioner in accelerating the convergence of GMRES \cite{Cao2016simplified}.
Decomposing $\mathcal{P}_{MHSS}$ as
\begin{equation*}
	\begin{split}
		\mathcal{P}_{MHSS}=\dfrac{1}{\alpha}\left(\begin{array}{cc}
			A& 0 \\
			0 & \alpha I\\
		\end{array}\right)\left(\begin{array}{cc}
			I& 0 \\
			\dfrac{1}{\alpha} M & I\\
		\end{array}\right)\left(\begin{array}{cc}
			\alpha I& 0 \\
			0 & B+\dfrac{1}{\alpha}M^2\\
		\end{array}\right)\left(\begin{array}{cc}
			I& -\dfrac{1}{\alpha} M \\
			0 & I\\
		\end{array}\right),
	\end{split}
\end{equation*}
then we obtain
\begin{align}
	\mathcal{T}(\alpha)
	&=\mathcal{P}_{MHSS}^{-1}\mathcal{R}\notag\\
	&=\begin{pmatrix}
		0& A^{-1}M-\frac{1}{\alpha}M+\frac{1}{\alpha}M(B+\frac{1}{\alpha}M^2)^{-1}(\frac{1}{\alpha}M^2-MA^{-1}M) \\
		0 & (B+\frac{1}{\alpha}M^2)^{-1}(\frac{1}{\alpha}M^2-MA^{-1}M)
	\end{pmatrix} \notag \\
	&=\begin{pmatrix}
		0& A^{-1}M-\frac{1}{\alpha}M+\frac{1}{\alpha}MD^{-1}F \\
		0 & D^{-1}F
	\end{pmatrix},\label{1.171}
\end{align}
where
\begin{align}
	D=B+\frac{1}{\alpha}M^2, ~~F=\frac{1}{\alpha}M^2-MA^{-1}M.\label{1.5}
\end{align}
It follows that
\begin{equation}
	\begin{split}
		\mathcal{P}_{MHSS}^{-1}\tilde{\mathcal{A}}=\mathcal{P}_{MHSS}^{-1}(\mathcal{P}_{MHSS}-\mathcal{R})&=I-\mathcal{P}_{MHSS}^{-1}\mathcal{R}\\
		&=I-\mathcal{T}(\alpha)\\
		&=\left(\begin{array}{cc}
			I & -A^{-1}M+\frac{1}{\alpha}M-\frac{1}{\alpha}MD^{-1}F \\
			0 & I-D^{-1}F\\
		\end{array}\right).
	\end{split}\label{1.18}
\end{equation}

Now, we analyze the spectral distribution for the
preconditioned matrix $\mathcal{P}_{MHSS}^{-1}\tilde{\mathcal{A}}$.

\begin{theorem}\label{th3}
	Let $A\in \mathbb{R}^{\kappa\times \kappa}$,
	$M\in \mathbb{R}^{\kappa\times \kappa}$ and $M\succ 0$,
	$B\in \mathbb{R}^{\kappa\times \kappa}$ and $B\succeq 0$,
	$\alpha$ be a positive constant. Then 
	$\mathcal{P}_{MHSS}^{-1}\tilde{\mathcal{A}}$ has at least $\kappa$ eigenvalues $1$.
	The remaining eigenvalues $\lambda$ are real and located in the positive interval
	\begin{align}
		\frac{\alpha c_{\kappa}^2}{\theta_1(c_1^2+\alpha\sigma_1)}
		\leq \lambda \leq \frac{\alpha (\sigma_1\theta_{\kappa}+c_1^2)}
		{\theta_{\kappa} c_{\kappa}^2},\label{1.19}
	\end{align}
where $c_{\kappa}$, $\theta_{\kappa}$ and
	$c_1$, $\theta_1$ are the minimum and maximum eigenvalues of $M,~A$, respectively. $\sigma_1$ is the maximum eigenvalue of $B$.
	Furthermore, if $A=\alpha I$, then all the eigenvalues of $\mathcal{P}_{MHSS}^{-1}\tilde{\mathcal{A}}$ are $1$.
\end{theorem}

\begin{proof}
From \eqref{1.18}, it is evident that $\mathcal{P}_{MHSS}^{-1}\tilde{\mathcal{A}}$
	has at least $\kappa$ eigenvalues $1$. The remaining are the eigenvalues of $I-D^{-1}F$.
	Let
	\begin{align*}
		I-D^{-1}F=D^{-1}(D-F)=(B+\frac{1}{\alpha}M^2)^{-1}(B+MA^{-1}M),
	\end{align*}
	$\lambda$ and $u$ be the eigenvalue and corresponding eigenvector of $I-D^{-1}F$, thus
	\begin{align}
		(B+MA^{-1}M)u=\lambda~(B+\frac{1}{\alpha}M^2)u.\label{1.20}
	\end{align}
	Premultiplying both sides of \eqref{1.20} by $u^{\ast}/(u^{\ast}u)$ gets
	\begin{align*}
		\frac{u^{\ast}(B+MA^{-1}M)u}{u^{\ast}u}=\lambda\frac{u^{\ast}(B+\frac{1}{\alpha}M^2)u}{u^{\ast}u}.
	\end{align*}
	Define
	\begin{align}
		d=\frac{u^{\ast}MA^{-1}Mu}{u^{\ast}u},
		~\gamma=\frac{u^{\ast}Bu}{u^{\ast}u},
		~\mbox{and}~\tau=\frac{u^{\ast}M^2u}{u^{\ast}u},\label{2.20}
	\end{align}
	then
	\begin{align}
		\lambda=\frac{{\alpha}(\gamma+d)}{{\alpha}\gamma+\tau}.\label{1.21}
	\end{align}
	Note that $A\succ 0,~M\succ 0$ and $B\succeq 0$, then
	\begin{align}
		0<\frac{c_{\kappa}^2}{\theta_1}\leq d \leq \frac{c_1^2}{\theta_{\kappa}},
		~0<c_{\kappa}^2 \leq \tau \leq c_{1}^2,
		~\mbox{and}~0\leq \gamma \leq \sigma_1.\label{1.22}
	\end{align}
	Combining \eqref{1.21} with \eqref{1.22}, we obtain \eqref{1.19}.
\end{proof}

\begin{remark}\label{re4}
Note that the positive interval presented in Theorem \ref{th3} is not very tight. 
	Therefore, choosing appropriate $\alpha$ could obtain a better spectral distribution. 
	This is illustrated by a simple numerical experiment in Fig. \ref{FIG:8}.
	\begin{figure}[htb]
		\centering
		\label{fig:v}
		\includegraphics[width=2.4in]{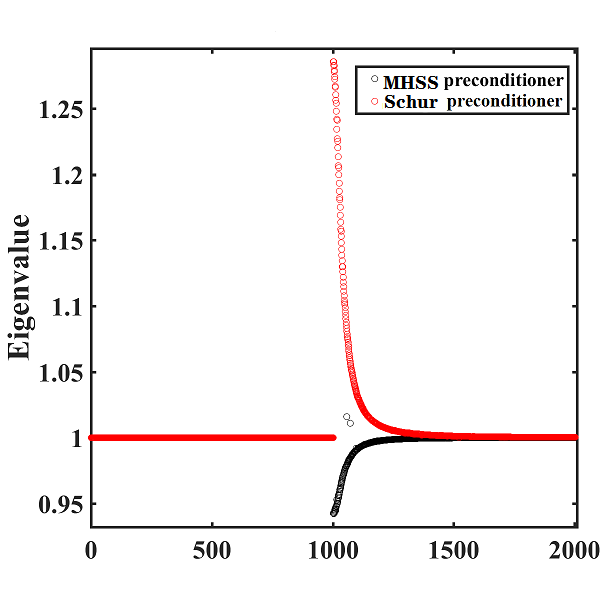}
		\caption{The spectral distribution of the preconditioned linear system with $\kappa=1000,~\epsilon=0.01,~dt=0.001,~\sigma=10,~ n=1,~k=1$, $\alpha=0.001$.}\label{FIG:8}
	\end{figure}
\end{remark}

In the following, we present an approach to get $\alpha$ such that $\rho(\mathcal{T}(\alpha))$ is as small as possible. However,
it is difficult to compute $\rho(\mathcal{T}(\alpha))$ exactly.
Instead, we give an upper bound of $\rho(\mathcal{T}(\alpha))$.

\begin{theorem}\label{th:min}
	Let $A\succ 0, ~ M\succ 0\in \mathbb{R}^{\kappa\times \kappa}$,
	$B\succeq 0\in \mathbb{R}^{\kappa\times \kappa}$, and
	$\mathcal{T}(\alpha)$ is defined in \eqref{1.171}. Then, we get
	\begin{align*}
		\rho(\mathcal{T}(\alpha))\leq\max_{1\leq i\leq \kappa}
		\Big\vert 1-\frac{\alpha}{\lambda_i(A)} \Big\vert =\tilde{\sigma}(\alpha).
	\end{align*}
\end{theorem}

\begin{proof}
	By the definition of $\mathcal{T}(\alpha)$ in \eqref{1.171}, it follows that
	\begin{align*}
		\rho(\mathcal{T}(\alpha))=\rho(D^{-1}F),
	\end{align*}
	where $D$ and $F$ are defined in \eqref{1.5}.
	Note that
	$$D^{-1}F=M^{-1}\tilde{D}^{-1}\tilde{F}M,$$
	where
	\begin{align*}
		\tilde{D}=\frac{1}{\alpha}I+M^{-1}BM^{-1}, ~~\tilde{F}=\frac{1}{\alpha}I-A^{-1}.
	\end{align*}
Therefore
	\begin{align*}
		\rho(\mathcal{T}(\alpha))=\rho(D^{-1}F)
		=\rho(\tilde{D}^{-1}\tilde{F})\leq \Vert \tilde{D}^{-1}\tilde{F}\Vert_2\leq \Vert\tilde{D}^{-1}\Vert_2 \Vert\tilde{F}\Vert_2.
	\end{align*}
Since $A\succ 0$, $M\succ0$ and $B \succeq 0$, we have
	\begin{align*}
		\Vert \tilde{D}^{-1}\Vert_2=\Big\Vert(\frac{1}{\alpha}I+M^{-1}BM^{-1})^{-1} \Big\Vert_2
		&=\max_{1\leq i\leq \kappa} \Big\vert \lambda_i[(\frac{1}{\alpha}I+M^{-1}BM^{-1})^{-1}] \Big\vert\\
		&=\max_{1\leq i\leq \kappa} \vert\frac{1}{\frac{1}{\alpha}+\lambda_i(M^{-1}BM^{-1})} \vert
		=\alpha,
	\end{align*}
	and
	\begin{align*}
		\Vert\tilde{F}\Vert_2= \Big\Vert\frac{1}{\alpha}I-A^{-1} \Big\Vert_2
		&=\max_{1\leq i\leq \kappa} \Big\vert\lambda_i(\frac{1}{\alpha}I-A^{-1}) \Big\vert\\
		&=\max_{1\leq i\leq \kappa} \Big\vert\frac{1}{\alpha}-\lambda_i(A^{-1}) \Big\vert.
	\end{align*}
	Hence
	\begin{align*}
		\rho(\mathcal{T}(\alpha))=\rho(D^{-1}F)\leq \alpha \max_{1\leq i\leq \kappa} \Big\vert\frac{1}{\alpha}-\lambda_i(A^{-1}) \Big\vert
		=\max_{1\leq i\leq \kappa} \Big\vert1-\frac{\alpha}{\lambda_i(A)} \Big\vert.
	\end{align*}
\end{proof}

\begin{remark}
	By Theorem \ref{th:min}, we can choose $\alpha$ as the optimal $\alpha_{*}$ to minimize $\tilde{\sigma}(\alpha)$
	\begin{align*}
		\alpha=\alpha_{*}=\mbox{arg}~ \min_{\alpha} \tilde{\sigma}(\alpha)
		=\frac{2\lambda_1(A)\lambda_{\kappa}(A)}{\lambda_1(A)+\lambda_{\kappa}(A)}.
	\end{align*}
	Since $\alpha_*$ is related to the eigenvalues of $A$, it is
	expensive to compute the optimal parameter $\alpha_*$ for
	large scale matrix $A$. Therefore, in practical implementation, we need to make an approximation for $\alpha_*$, refer to \cite{H.2014}.
	The practical choice strategy
	for parameter $\alpha$ is
	\begin{align*}
		\alpha=\frac{\mbox{trace}(A)}{\kappa} ~~\mbox{or}~~ \alpha=\frac{\mbox{trace}(M^4)}{\mbox{trace}(M^4A^{-1})}.
	\end{align*}
\end{remark}

\section{Numerical results}
\label{sec:rslts}

In this section, we offer several examples to support the theoretical analysis.
The whole process is performed on a computer with Intel Core 3.20GHz CPU, 4.00GB
RAM and MATLAB R2017a. 
The piecewise linear functions are used as the basis functions.
GMRES is adopted to solve the preconditioned linear system \eqref{eq:discretedmatrix} in each iteration. The elapsed CPU time per nonlinear iteration process in seconds is
denoted by `CPU' and the average of GMRES iteration per time step by IT$_{GM}$.
Let $k_n$ be the iteration steps of the $n$-th nonlinear iteration, and
IT$_{tol}$ be the total iteration steps after $n$-th nonlinear iteration, i.e.,
$\mbox{IT}_{tol}=\sum_{i=1}^{n} k_i$. The nonlinear iteration error is defined
as $\varepsilon = \Vert u^{n,k}-u^{n,k-1}\Vert_2$. The stop criteria $\varepsilon$
is chosen as $10^{-6}$.

First, we demonstrate the performance of Newton method and Picard method by solving the linear system \eqref{eq:discretedmatrix} directly.
When the total number of time steps is $N=100$, Tab.\,\ref{tab:1} compares
these two methods against CPU and $\mbox{IT}_{tol}$ in one and two
dimensional space with different degrees of freedom (DOF).
Newton method obviously requires less CPU and $\mbox{IT}_{tol}$ than Picard method.
\begin{table}[!hptb]
	\centering
	\caption{A comparison of two Netwon methods with different DOF
		for domains $\Omega=(0,1)$ and $\Omega=(0,\pi)^2$.}
	\label{tab:1}
	\renewcommand\arraystretch{1.2}
	\setlength{\tabcolsep}{1.2mm}
	{\begin{tabular}{cccc|cccc}\hline
			$\{\Omega,~~ \epsilon\}$&         &$\{(0,1), ~~0.01\}$		&   &        & $\{(0,\pi)^2, ~~0.001\}$ &      &     \\ \hline
			DOF                      & Index & Newton     &Picard         &DOF & Index              & Newton  &Picard  \\ \hline
			\multirow{2}{*}{$10^4$} &CPU(s)      & 15.56  &57.12   &\multirow{2}{*}{$10^4$}&CPU  & 52.54      &84.76    \\
			&IT$_{tol}$  &346     &1359     &        &IT$_{tol}$          &300        & 486 \\ \cline{1-8}
			\multirow{2}{*}{$4\times 10^5$}	&CPU &683.30 &2371.97 &\multirow{2}{*}{$4\times10^4$} &CPU   &301.78 &750.69  \\
			&IT$_{tol}$ &327    &1183        &           &IT$_{tol}$        &300             &753\\ \hline
	\end{tabular}}
\end{table}

Furthermore, Fig.\,\ref{fig:1} gives the concrete iteration process when $\Omega=(0,1)$, DOF$=10^{4}$, $\sigma=100$, $\Delta t =
\epsilon^2$ and $\epsilon=0.01$.
Note that CPU time required for each nonlinear iteration
step of these two methods is almost the same.
Fig.\,\ref{fig:subfig:b} represents that Newton method
requires fewer nonlinear iteration steps than that of Picard method for each time step $n$ and CPU time. When $n=23$, Fig.\,\ref{fig:subfig:a} illustrates that
$\varepsilon$ of Newton method decreases faster than
that of Picard method. Obviously, Newton method converges quadratically and Picard method converges linearly, as the theory in Sec.\,\ref{sec:CA} predicted.

\begin{figure}[!htbp]
	\centering
	\subfigure[CPU time against iterations]{
		\label{fig:subfig:b}
		\includegraphics[width=2in]{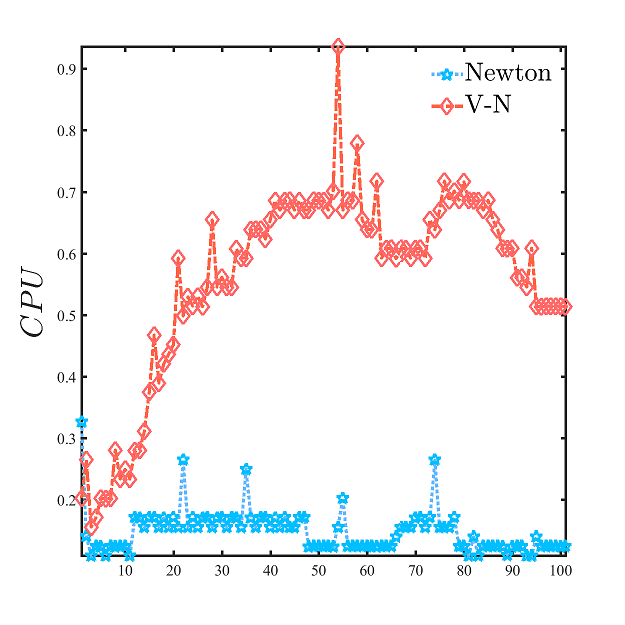}
	}
	\hspace{0.15in}
	\subfigure[Error tendency of the $23rd$ time step]{
		\label{fig:subfig:a}
		\includegraphics[width=2in]{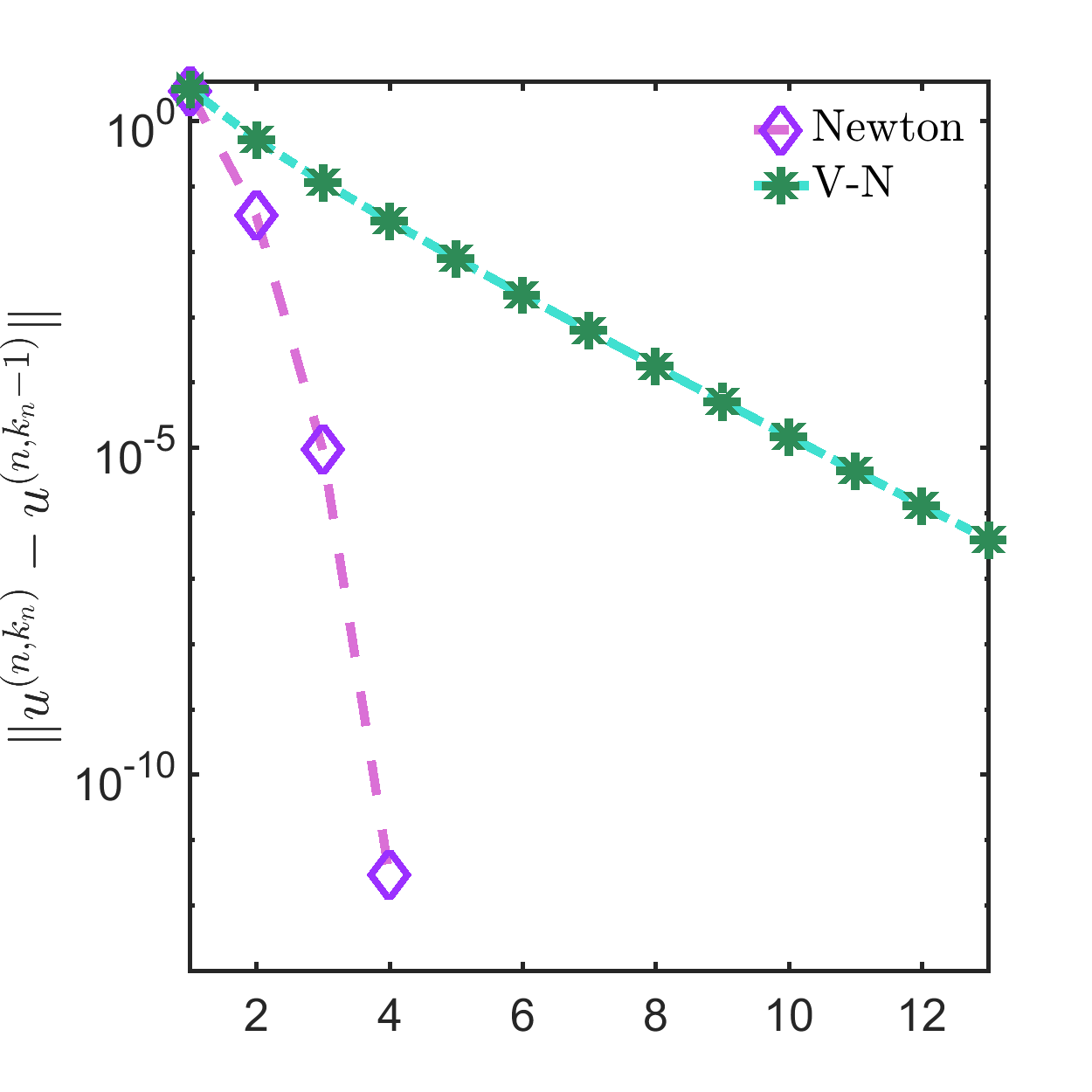}
	}
	\caption{Comparison of convergence speed between Newton method and Picard method with $\Omega=(0,1)$ and DOF$=10^4$.}\label{fig:1}
\end{figure}

Second, we demonstrate the efficiency of proposed preconditioners when
solving linear system \eqref{eq:discretedmatrix}.
As a one-dimensional example, Table \ref{tab:2} presents the performance of solving
non-preconditioned and preconditioned systems
in $500$ time steps ($\Delta t=\epsilon^2$) with different spatial DOF
for domains $\Omega=(0,1)$.
From these results, it can be seen that two preconditioners
accelerate computation, almost three times faster than solving the non-preconditioned system in terms of CPU time.

\begin{table}[htb]
	\centering
	\caption{A comparison of computing non-preconditioned and preconditioned systems
		in $500$ time steps ($\Delta t=\epsilon^2$) with different spatial DOF and $\Omega=(0,1)$.}\label{tab:2}
	\renewcommand\arraystretch{1.2}
	\setlength{\tabcolsep}{4mm}{
		\begin{tabular}{|c|ccccc|}\hline			
			$DOF$ 
			&Index          &($\epsilon$, $\sigma$)   &IT$_{GM}$        & IT$_{tol}$          &CPU (s)\\ \hline
			\multirow{6}*{$10^4$ }	&Newton               &                         &  -           &              &263.15 \\
			&Newton-$\mathcal{P}_{Schur}$        &  (0.001,100)             &5.5          &1500         &101.97\\
			&Newton-$\mathcal{P}_{MHSS}$        &                          &14.5          &             &107.96\\ \cline{2-6}
			&Picard                  &                          &  -          &            &422.81 \\
			&Picard-$\mathcal{P}_{Schur}$            &   (0.001,100)            &5.0          &2349           &156.65\\
			&Picard-$\mathcal{P}_{MHSS}$             &                          &7.5          &              &165.27\\ \hline
			\multirow{6}*{$4\times 10^4$}	&Newton             &                          &  -         &             &1539.59 \\
			&Newton-$\mathcal{P}_{Schur}$            &  (0.001,100)             &7.0        &1500         &592.27\\
			&Newton-$\mathcal{P}_{MHSS}$          &                          &8.4         &              &580.30\\ \cline{2-6}
			&Picard           &                          &-           &              &3553.33 \\
			&Picard-$\mathcal{P}_{Schur}$              &   (0.001,100)            &6.5         &3291             &1271.09\\
			&Picard-$\mathcal{P}_{MHSS}$            &                          &16.9         &             &1267.39\\ \hline
	\end{tabular}}\\
\end{table}

Last, we observe the long-time behavior of the proposed method in the domain
$\Omega=(0,\pi)^2$ with $1000\times 1000$ triangular grids.
Using random initial values, Figs. \ref{fig:5} and \ref{fig:6} show the
coarsening dynamic processes with different parameters.
The final converged morphologies are lamellar and cylindrical phases, respectively.
As shown in these figures, our algorithm keeps energy dissipation as time evolves.

\begin{figure}[htb]
	\centering
	\includegraphics[width=0.75\textwidth]{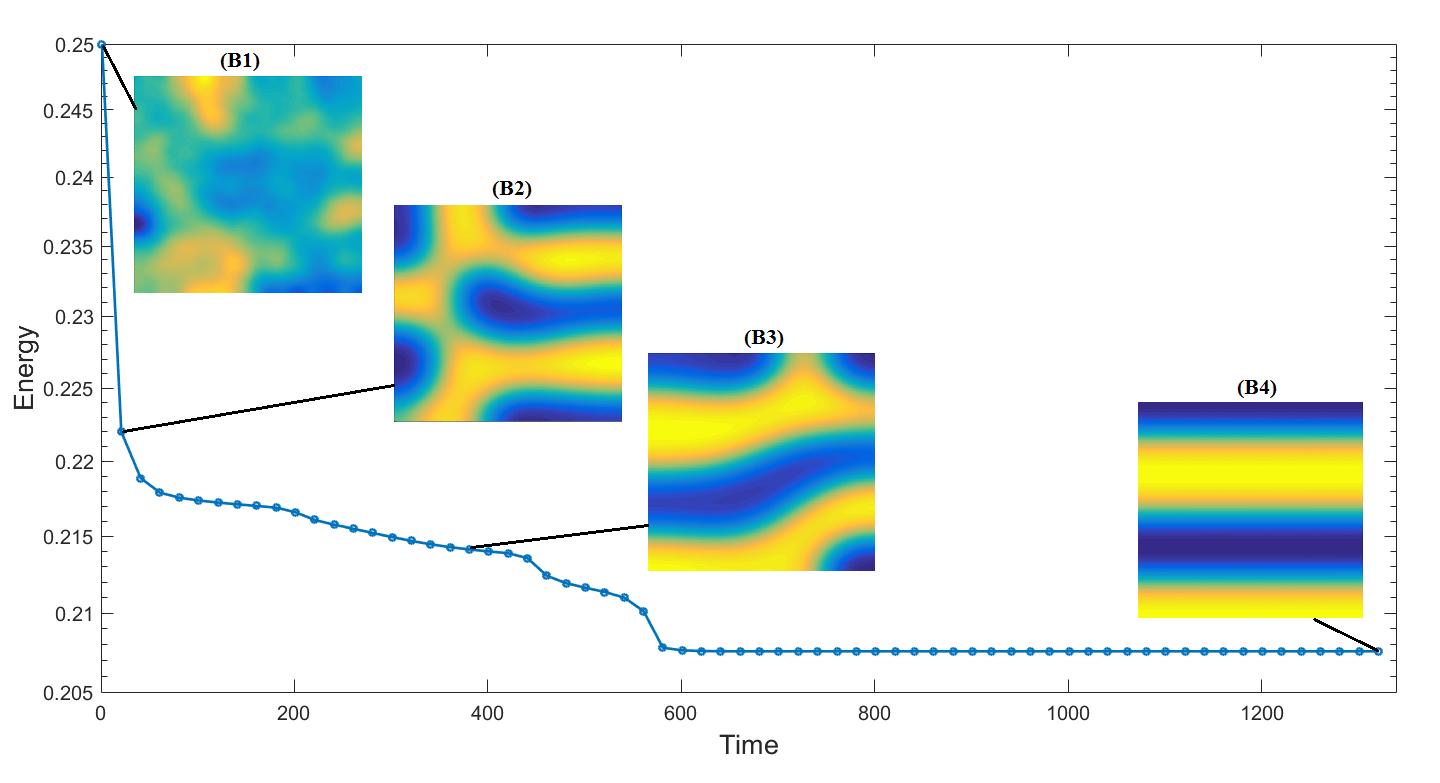}
	\caption{The dynamic process in domain $(0,\pi)^2$ when
		$\epsilon=0.08$, $\sigma=10$ , $m=0$ and $\Delta t=0.01$ with a random initial value.
		(B1), (B2), (B3) and (B4) correspond to $t=0, 0.2, 4, 13.43$.}
	\label{fig:5}
\end{figure}

\begin{figure}[htb]
	\centering
	\includegraphics[width=0.75\textwidth]{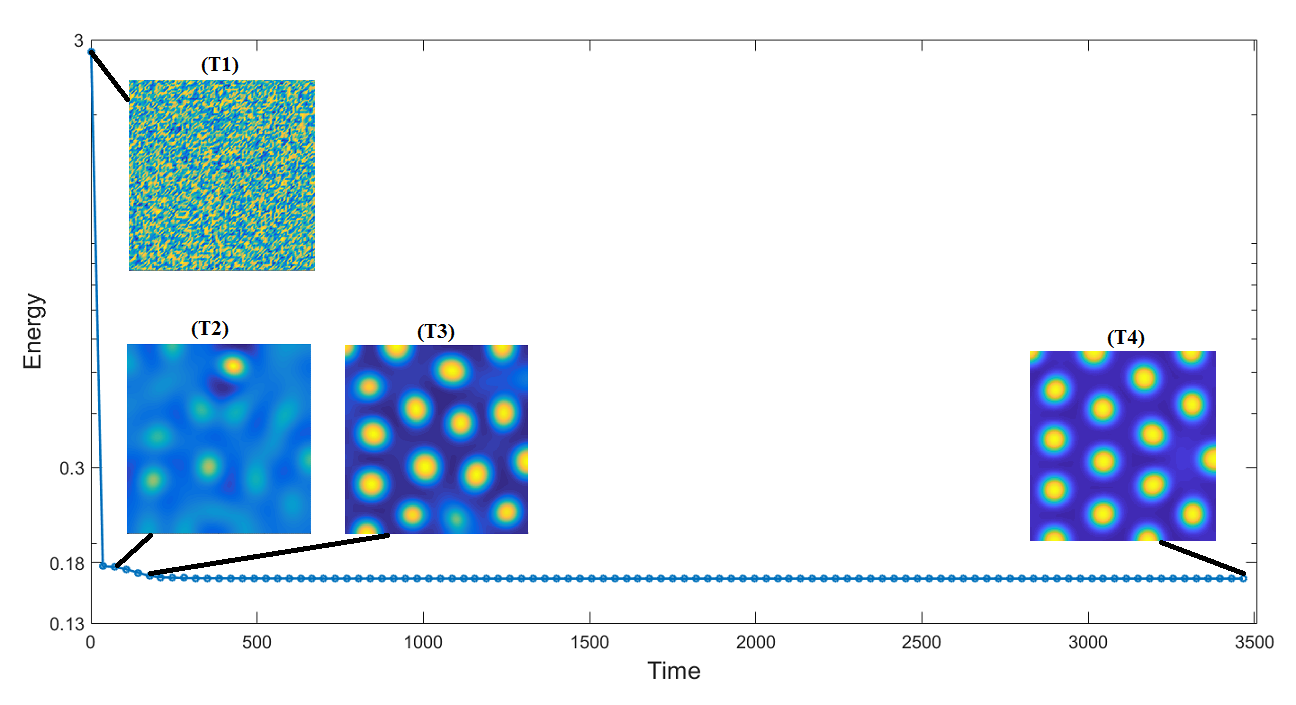}
	\label{fig:6}
	\caption{The dynamic process in domain $(0,\pi)^2$ when
		$\epsilon=0.08$, $\sigma=10$ , $m=0.4$ and $\Delta t=0.0064$ with a random initial value.
		(T1), (T2), (T3) and (T4) correspond to $t=0, 0.448, 1.344, 22.4$.}
\end{figure}

\section{Conclusion}
\label{sec:conclusion}

This paper presents a systematic numerical method to solve the mass-conserved
Ohta-Kawasaki equation. An unconditionally energy stable scheme, the CSS, is used to
discretize the time variable, while FEM is used for the spatial variable.
Newton method and Picard method are applied to update the implicit nonlinear terms.
To reduce the condition number of the
discretized linear system, we propose Schur complement preconditioner and MHSS
block preconditioner. The convergent rate of Newton method is proved.
The spectral distribution of two block preconditioning linear systems has been analyzed.
Numerical investigations provided sufficient support for the theoretical
analysis and demonstrated the efficiency of the proposed numerical methods.





\section*{Appendix}

\begin{appendices}
	
\section{Proof of Theorem \ref{thm:dissp}}\label{secA1}
	
	\begin{proof}
		Using the definition of $E(u)$ in \eqref{eq:okenergy}, then
		\begin{align}
			E(u^{n})-E(u^{n-1})
			&=\frac{1}{2}\epsilon^2(\Vert \nabla u^n\Vert ^2-\Vert \nabla u^{n-1}\Vert ^2)+(\Phi(u^n)-\Phi(u^{n-1}),1)\nonumber\\
			&~~+\frac{1}{2}\sigma(\Vert (-\Delta)^{-\frac{1}{2}}(u^n-m)\Vert ^2-\Vert 
			(-\Delta)^{-\frac{1}{2}}(u^{n-1}-m)\Vert ^2).
			\label{dissp:eq1}
		\end{align}
		Using Taylor expansion, we have
		\begin{align*}
			&\Phi_{+}(u^{n-1})-\Phi_{+}(u^n)=\Phi'_{+}(u^n)(u^{n-1}-u^n)+\frac{\Phi_{+}^{''}(\xi^{n-1})}{2}(u^{n-1}-u^n)^2,\\
			&\Phi_{-}(u^{n})-\Phi_{-}(u^{n-1})=\Phi'_{-}(u^{n-1})(u^{n}-u^{n-1})+\frac{\Phi_{-}^{''}(\eta^{n-1})}{2}(u^{n}-u^{n-1})^2.
		\end{align*}
		Thus,
		\begin{align}
			\nonumber\Phi(u^n)-\Phi(u^{n-1})
			&=-(\Phi_{+}(u^{n-1})-\Phi_{+}(u^n))-(\Phi_{-}(u^{n})-\Phi_{-}(u^{n-1}))\\
			&=(\Phi'_{+}(u^n)-\Phi'_{-}(u^{n-1}))(u^{n}-u^{n-1}) \notag\\
			&-\frac{1}{2}(\Phi_{+}^{''}(\xi^{n-1})+\Phi_{-}^{''}
			(\eta^{n-1}))(u^{n}-u^{n-1})^2.
			\label{dissp:eq2}
		\end{align}
		Combing the identity
		\begin{align*}
			a^2-b^2=2a(a-b)-(a-b)^2
		\end{align*}
		with \eqref{dissp:eq2}, \eqref{dissp:eq1} is equivalent to
		\begin{align}
			\nonumber&E(u^{n})-E(u^{n-1})\\
			\nonumber&=\epsilon^2(\nabla u^n,\nabla (u^n-u^{n-1}))-\frac{1}{2}\epsilon^2\Vert \nabla (u^{n}-u^{n-1})\Vert^2\\
			\nonumber&~~+(\Phi'_{+}(u^n)-\Phi'_{-}(u^{n-1}),u^n-u^{n-1})-\frac{1}{2}(\Phi_{+}^{''}(\xi^{n-1})+\Phi_{-}^{''}(\eta^{n-1}),(u^{n}-u^{n-1})^2)\\
			\nonumber&~~+\sigma((-\Delta)^{-\frac{1}{2}}(u^n-m),(-\Delta)^{-\frac{1}{2}}(u^n-u^{n-1}))-\frac{1}{2}\sigma\Vert (-\Delta)^{-\frac{1}{2}}(u^n-u^{n-1})\Vert^2\\
			\nonumber&=(-\epsilon^2\Delta u^n,u^n-u^{n-1})-\frac{1}{2}\epsilon^2\Vert \nabla (u^{n}-u^{n-1})\Vert^2\\
			\nonumber&~~+(\Phi'_{+}(u^n)-\Phi'_{-}(u^{n-1}),u^n-u^{n-1})-\frac{1}{2}(\Phi_{+}^{''}(\xi^{n-1})+\Phi_{-}^{''}(\eta^{n-1}),(u^{n}-u^{n-1})^2)\\
			\nonumber&~~+\sigma((-\Delta)^{-1}(u^n-m),u^n-u^{n-1})-\frac{1}{2}\sigma\Vert (-\Delta)^{-\frac{1}{2}}(u^n-u^{n-1})\Vert^2\\
			\nonumber&=(-\epsilon^2\Delta u^n+\Phi'_{+}(u^n)-\Phi'_{-}(u^{n-1})+\sigma(-\Delta)^{-1}(u^n-m),u^n-u^{n-1})\\
			\nonumber&~~-\frac{1}{2}\epsilon^2\Vert \nabla (u^{n}-u^{n-1})
			\Vert^2-\frac{1}{2}(\Phi_{+}^{''}(\xi^{n-1})+\Phi_{-}^{''}(\eta^{n-1}),(u^{n}-u^{n-1})^2)\\
			&~~-\frac{1}{2}\sigma\Vert (-\Delta)^{-\frac{1}{2}}(u^n-u^{n-1})\Vert^2.
			\label{dissp:eq3}
		\end{align}
		Taking the test function $v=u^n-u^{n-1}$ in \eqref{eq:semidiscretedCH}, then
		\begin{subequations}
			\begin{align}
				\label{dissp:eq4.1}
				&(\frac{u^n-u^{n-1}}{\Delta t},u^n-u^{n-1})+(\nabla w^n, \nabla (u^{n}-u^{n-1}))+\sigma(u^n-m ,u^n-u^{n-1})=0,\\
				\label{dissp:eq4.2}
				&(w^n,u^n-u^{n-1})-\epsilon^2(\nabla u^n, \nabla
				(u^n-u^{n-1}))-(\Phi'_{+}(u^n)-\Phi'_{-}(u^{n-1}),u^n-u^{n-1})=0.
			\end{align}
		\end{subequations}
		Putting
		\begin{align*}
			\mu^n=w^n+\sigma (-\Delta)^{-1}(u^n-m)
		\end{align*}
		into \eqref{dissp:eq4.1}, rearranging \eqref{dissp:eq4.2}, thus
		\begin{subequations}
			\begin{align}
				\label{dissp:eq5.1}
				&(\frac{u^n-u^{n-1}}{\Delta t},u^n-u^{n-1})=(\Delta \mu^n,u^n-u^{n-1}),\\
				\label{dissp:eq5.2}
				&(w^n,u^n-u^{n-1})=-\epsilon^2(\Delta
				u^n,u^n-u^{n-1})+(\Phi'_{+}(u^n)-\Phi'_{-}(u^{n-1}),u^n-u^{n-1}).
			\end{align}
		\end{subequations}
		Note that
		$$\Phi_{+}(u)\geq 0, ~~\Phi_{-}(u)\geq 0.$$
		Substituting\eqref{dissp:eq5.1} and \eqref{dissp:eq5.2} into \eqref{dissp:eq3} yields
		\begin{align*}
			&E(u^{n})-E(u^{n-1})\\
			&=(\mu^n,u^n-u^{n-1})-\frac{1}{2}\epsilon^2\Vert \nabla (u^{n}-u^{n-1})\Vert^2\\
			&~~-\frac{1}{2}(\Phi_{+}^{''}(\xi^{n-1})
			+\Phi_{-}^{''}(\eta^{n-1}),(u^{n}-u^{n-1})^2)
			-\frac{1}{2}\sigma\Vert (-\Delta)^{-\frac{1}{2}}(u^n-u^{n-1})\Vert^2.
		\end{align*}
		By the definition of $\mu^n$ and let the test function $v=\mu^n$ in \eqref{eq:semidiscretedCH}, we can obtain
		\begin{align*}
			(\mu^n,u^n-u^{n-1})=\Delta t (\mu^n,\Delta \mu^n),
		\end{align*}
		then
		\begin{align*}
			&E(u^{n})-E(u^{n-1})\\
			&=\Delta t (\mu^n,\Delta \mu^n)-\frac{1}{2}\epsilon^2\Vert \nabla (u^{n}-u^{n-1})\Vert^2\\
			&~~-\frac{1}{2}(\Phi_{+}^{''}(\xi^{n-1})
			+\Phi_{-}^{''}(\eta^{n-1}),(u^{n}-u^{n-1})^2)
			-\frac{1}{2}\sigma\Vert (-\Delta)^{-\frac{1}{2}}(u^n-u^{n-1})\Vert^2
			\leq 0,
		\end{align*}
		which implies \eqref{eq:energydissp}. Thus we complete the proof of Theorem \ref{thm:dissp}.
	\end{proof}
	
	\section{Proof of Theorem \ref{thm:bound}}\label{secA2}
	
	\begin{lemma}\label{le1}
		\cite{E.2010} (Poincar\'{e} inequality) Let $\Omega$ be a bounded, connected, open subset of $\mathbb{R}^{d}$ with boundary $\partial\Omega$ of $\mathbb{C}^1$. Then there exists a positive constant $c$, depending only on $d$ and $\Omega$ such that
		\begin{align*}
			\Vert v-\int_{\Omega}v\,dx \Vert \leq c \Vert \nabla v\Vert.
		\end{align*}
	\end{lemma}
	
	\begin{definition}\cite{P.Q.2012}
		For $u\in H^1(\Omega)$, we define $\Delta_h u \in V_h$ such that
		\begin{align*}
			(-\Delta_h u,v_h)=(\nabla u,\nabla v_h),~~\forall v_h\in V_h,
		\end{align*}
		where the subscript `h' means that $\Delta_h u$ depends on the discretization mesh.
	\end{definition}
	
	\begin{proof}
		From Poincar\'{e} inequality (Lemma \ref{le1}), we have
		\begin{align*}
			\Vert u^n_h-m\Vert
			=\Vert u^n_h-\bbint_{\Omega}u^n_h\,dx \Vert
			\leq c \Vert\nabla u^n_h\Vert.
		\end{align*}
		Therefore
		\begin{align}
			\Vert u_h^n\Vert-m\vert \Omega\vert \leq \Vert u_h^n-m\Vert \leq c \Vert \nabla u^n_h\Vert .
			\label{thm:eq0}
		\end{align}
		Set $v_h=w_h^n$ in \eqref{eq:fullCH-1}, then
		\begin{align}
			\Big(\frac{u^n_h-u^{n-1}_h}{\Delta t},w_h^n \Big)+(\nabla w^n_h, \nabla w_h^n)+\sigma(u^n_h - m ,w_h^n)=0.\label{thm:eq1}
		\end{align}
		Taking $v_h=(u^n_h-u^{n-1}_h)/\Delta t$ in \eqref{eq:fullCH-2} yields
		\begin{align}
			(w^n_h,\frac{u^n_h-u^{n-1}_h}{\Delta t})
			=\epsilon^2(\nabla u^n_h, \nabla\frac{u^n_h-u^{n-1}_h}{\Delta t})+(\Phi'_+(u^n_h)-\Phi'_-(u^{n-1}_h), \frac{u^n_h-u^{n-1}_h}{\Delta t}).\label{thm:eq2}
		\end{align}
		Subtracting \eqref{thm:eq1} from \eqref{thm:eq2}, using the symmetry
		in the inner product, thus
		\begin{align}
			\epsilon^2\Big(\nabla u^n_h, \frac{\nabla u^n_h-\nabla u^{n-1}_h}
			{\Delta t}\Big)+\Big(\Phi'_+(u^n_h)&-\Phi'_-(u^{n-1}_h), \frac{u^n_h-u^{n-1}_h}{\Delta t}\Big)\notag\\
			&+\Vert\nabla w^n_h\Vert^2+\sigma(u^n_h - m ,w_h^n)=0.\label{thm:eq3}
		\end{align}
		Note that
		$$\nabla u^n_h=\frac{\nabla u^n_h-\nabla u^{n-1}_h}{2}+\frac{\nabla u^n_h+\nabla u^{n-1}_h}{2},$$
		and \eqref{thm:eq3} is equivalent to
		\begin{align}
			\frac{\epsilon^2}{2\Delta t}&\Big{[}(\nabla (u^n_h-u^{n-1}_h), \nabla (u^n_h-u^{n-1}_h))+(\nabla u^n_h+\nabla u^{n-1}_h, \nabla u^n_h- \nabla u^{n-1}_h)\Big{]}\notag\\
			&+\Big{(}\Phi'_+(u^n_h)-\Phi'_-(u^{n-1}_h), \frac{u^n_h-u^{n-1}_h}{\Delta t}\Big{)}+\Vert\nabla w^n_h\Vert^2+\sigma(u^n_h - m ,w_h^n)=0.\label{thm:eq4}
		\end{align}
		The sequence $u_h^n$ keeps mass conservation in \cite{P.Q.2012}, \textit{i.e.,}
		\begin{align}
			(u_h^n-m,1)=0 ~~\mbox{for}~~ n=0,1,\cdots,N.\label{thm:eq5}
		\end{align}
		Multiplying \eqref{thm:eq5} by $\bbint_{\Omega} w_h^{n}\,dx\in \mathbb{R}$ gets
		$$(u_h^n-m,\bbint_{\Omega} w_h^{n}\,dx)=0 .$$
		Hence,
		\begin{align}
			(u_h^n-m, w_h^{n})=(u_h^n-m, w_h^n-\bbint_{\Omega} w_h^{n}\,dx).\label{thm:eq6}
		\end{align}
		Combining \eqref{thm:eq4} with \eqref{thm:eq6}, we have
		\begin{align}
			&\frac{\epsilon^2}{2\Delta t}\Big{\{}\Vert\nabla (u^n_h-u^{n-1}_h)\Vert^2+\Vert \nabla u^n_h\Vert^2-\Vert\nabla u^{n-1}_h\Vert^2\Big{\}}+\frac{1}{\Delta t}(\Phi'_+(u^n_h)-\Phi'_-(u^{n-1}_h),u^n_h-u^{n-1}_h)\notag\\
			&~~~~~~~~ +\Vert\nabla w^n_h\Vert^2+\sigma(u^n_h-m ,w_h^n-\bbint_{\Omega} w_h^{n}\,dx)=0.\label{thm:eq7}
		\end{align}
		Using Taylor expansion for $\Phi(b)$ at $a$, then
		$$\Phi'(a)(a-b)=\Phi(a)-\Phi(b)+\frac{1}{2}\Phi''(\eta)(b-a)^2,~~\eta\in(a,b).$$
		Denote
		$$\Phi'_{+}(u)=u^3, ~~\Phi'_{-}(u)=u, ~~\Phi''_{+}(u)=3u^2> -1, ~~\Phi''_{-}(u)=1.$$
		Hence,
		\begin{subequations}
			\begin{align}
				\label{thm:eq8.1}
				&(\Phi_{+}(u_h^{n}),1)-(\Phi_{+}(u_h^{n-1}),1)-\frac{1}{2}\Vert u_h^n-u_h^{n-1}\Vert^2
				< (\Phi'_{+}(u_h^n),u_h^n-u_h^{n-1}),\\
				\label{thm:eq8.2}
				&(\Phi_{-}(u_h^{n-1}),1)-(\Phi_{-}(u_h^{n}),1)+\frac{1}{2}\Vert u_h^n-u_h^{n-1}\Vert^2
				=(\Phi'_{-}(u_h^{n-1}),u_h^{n-1}-u_h^{n}).
			\end{align}
		\end{subequations}
		Substituting \eqref{thm:eq8.1} and \eqref{thm:eq8.2} in \eqref{thm:eq7}, multiplying by $2\Delta t$, we have
		\begin{align}
			\nonumber&\epsilon^2\Vert \nabla u^n_h\Vert^2+\epsilon^2\Vert\nabla (u^n_h-u^{n-1}_h)\Vert^2+2(\Phi_+(u^n_h),1)+2(\Phi_-(u^{n-1}_h),1)+2\Delta t\Vert \nabla w^n_h\Vert^2\\
			\nonumber&<\epsilon^2\Vert\nabla u^{n-1}_h\Vert^2+2(\Phi_+(u^{n-1}_h),1)+2(\Phi_-(u^{n}_h),1)-2\Delta t\sigma(u^n_h-m ,w_h^n-\bbint_{\Omega} w_h^{n}\,dx)\\
			\nonumber&<\epsilon^2\Vert\nabla u^{n-1}_h\Vert^2+2(\Phi_+(u^{n-1}_h),1)+2(\Phi_-(u^{n}_h),1)+2\Delta t\sigma\Vert u^n_h-m\Vert\cdotp \Vert w_h^n-\bbint_{\Omega} w_h^{n}\,dx\Vert\\
			&<\epsilon^2\Vert\nabla u^{n-1}_h\Vert^2+2(\Phi_+(u^{n-1}_h),1)+2(\Phi_-(u^{n}_h),1)+2\Delta t\sigma c^2\Vert\nabla u^n_h\Vert \cdotp\Vert \nabla w^n_h\Vert.
			\label{thm:eq9}
		\end{align}
		Applying the inequality $2ab\leq a^2+b^2$
		to the last term on the right-hand side of \eqref{thm:eq9}, subtracting $\Delta t\Vert\nabla w^n_h\Vert^2$ from both sides of \eqref{thm:eq9}, then
		\begin{align}
			\nonumber&\epsilon^2\Vert\nabla u^n_h\Vert^2+\epsilon^2\Vert \nabla (u^n_h-u^{n-1}_h)\Vert^2+2(\Phi_+(u^n_h),1)+2(\Phi_-(u^{n-1}_h),1)+\Delta t\Vert\nabla w^n_h\Vert^2\\
			&<\epsilon^2\Vert\nabla u^{n-1}_h\Vert^2+2(\Phi_+(u^{n-1}_h),1)+2(\Phi_-(u^{n}_h),1)+\sigma^2 c^4 \Delta t\Vert\nabla u^n_h\Vert^2.
			\label{thm:eq10}
		\end{align}
		Summing \eqref{thm:eq10} over $n=1,2,\cdots,\tilde{k}$ ($\tilde{k}\leq N$) yields
		\begin{align}
			\nonumber&\epsilon^2\sum_{n=1}^{\tilde{k}}\Vert\nabla u^n_h\Vert^2+\epsilon^2\sum_{n=1}^{\tilde{k}}\Vert\nabla (u^n_h-u^{n-1}_h)\Vert^2+2\sum_{n=1}^{\tilde{k}}(\Phi_+(u^n_h),1) \\
			&+2\sum_{n=1}^{\tilde{k}}(\Phi_-(u^{n-1}_h),1)+\Delta t\sum_{n=1}^{\tilde{k}}\Vert\nabla w^n_h\Vert^2 \notag \\
			&<\epsilon^2\sum_{n=1}^{\tilde{k}}\Vert\nabla u^{n-1}_h\Vert^2+2\sum_{n=1}^{\tilde{k}}(\Phi_+(u^{n-1}_h),1)+2\sum_{n=1}^{\tilde{k}}(\Phi_-(u^{n}_h),1)+\sigma^2 c^4 \Delta t\sum_{n=1}^{\tilde{k}}\Vert\nabla u^n_h\Vert^2.
			\label{thm:eq11}
		\end{align}
		Canceling the same terms on each side of \eqref{thm:eq11}, then
		\begin{align}
			\nonumber&\epsilon^2\Vert\nabla u^{\tilde{k}}_h\Vert^2+\epsilon^2\sum_{n=1}^{\tilde{k}}\Vert\nabla (u^n_h-u^{n-1}_h)\Vert^2+2(\Phi_+(u^{\tilde{k}}_h)-\Phi_-(u^{\tilde{k}}_h),1)
			+\Delta t\sum_{n=1}^{\tilde{k}}\Vert\nabla w^n_h\Vert^2\\
			&<\epsilon^2\Vert\nabla u^0_h\Vert^2+2(\Phi_+(u^0_h)-\Phi_-(u^0_h),1)+\sigma^2 c^4 \Delta t\sum_{n=1}^{\tilde{k}}\Vert\nabla u^n_h\Vert^2.\label{thm:eq12}
		\end{align}
		For $\tilde{k}=1,2,\cdots,N$, note that
		$$\Phi(u^{\tilde{k}}_h)=\Phi_+(u_h^{\tilde{k}})-\Phi_-(u_h^{\tilde{k}})\geq 0.$$
		Hence, \eqref{thm:eq12} reduces to
		\begin{align*}
			\epsilon^2\Vert\nabla u^{\tilde{k}}_h\Vert^2
			<\epsilon^2\Vert\nabla u^0_h\Vert^2+2(\Phi_+(u^0_h)-\Phi_-(u^0_h),1)+\sigma^2 c^4 \Delta t\sum_{n=1}^{\tilde{k}}\Vert\nabla u^n_h\Vert^2.
		\end{align*}
		Therefore,
		\begin{align}
			(\epsilon^2-\sigma^2 c^4 \Delta t)\Vert\nabla u^{\tilde{k}}_h\Vert^2
			<\epsilon^2\Vert\nabla u^0_h\Vert^2+2(\Phi_+(u^0_h)-\Phi_-(u^0_h),1)+\sigma^2 c^4 \Delta t\sum_{n=1}^{\tilde{k}-1}\Vert\nabla u^n_h\Vert^2.
			\label{thm:eq13}
		\end{align}
		Suppose
		\begin{align*}
			\frac{\epsilon^2}{2}\leq \epsilon^2-\sigma^2 c^4 \Delta t,
		\end{align*}
		by \eqref{thm:eq13}, we can obtain
		\begin{align*}
			\frac{\epsilon^2}{2}\Vert\nabla u^{\tilde{k}}_h\Vert^2
			<\epsilon^2\Vert\nabla u^0_h\Vert^2+2(\Phi_+(u^0_h)-\Phi_-(u^0_h),1)+\sigma^2 c^4 \Delta t\sum_{n=1}^{\tilde{k}-1}\Vert\nabla u^n_h\Vert^2.
		\end{align*}
		It is equivalent to
		\begin{align}
			\nonumber\Vert\nabla u^{\tilde{k}}_h\Vert^2
			\nonumber&<2\Vert\nabla u^0_h\Vert^2+\frac{4}{\epsilon^2}(\Phi(u^0_h),1)+\frac{2\sigma^2 c^4}{\epsilon^2}\Delta t\sum_{n=1}^{\tilde{k}-1}\Vert\nabla u^n_h\Vert^2\\
			&=\mathcal{K}^*+\mathcal{L}^*\Delta t\sum_{n=1}^{\tilde{k}-1}\Vert\nabla u^n_h\Vert^2,\label{thm:eq14}
		\end{align}
		where
		\begin{align*}
			\mathcal{K}^*=2\Vert\nabla u^0_h\Vert^2+\frac{4}{\epsilon^2}(\Phi(u^0_h),1)\geq 0,~~\mathcal{L}^*=\frac{2\sigma^2 c^4}{\epsilon^2}\geq 0.
		\end{align*}
		Denote
		\begin{align*}
			\kappa_{\tilde{k}}=\mathcal{K}^*+\mathcal{L}^*\Delta t\sum_{n=1}^{\tilde{k}-1}\Vert\nabla u^n_h\Vert^2.
		\end{align*}
		Applying \eqref{thm:eq14}, it follows that
		\begin{align*}
			\hat{\kappa}_{\tilde{k}+1}-\hat{\kappa}_{\tilde{k}}&=\mathcal{L}^*\Delta t\Vert\nabla u_h^{\tilde{k}}\Vert^2\\
			&<\mathcal{L}^*\Delta t(\mathcal{K}^*+\mathcal{L}^*\Delta t\sum_{n=1}^{\tilde{k}-1}\Vert\nabla u^n_h\Vert^2)\\
			&=\mathcal{L}^*\Delta t\hat{\kappa}_{\tilde{k}}.
		\end{align*}
		Hence, we have
		\begin{align*}
			\hat{\kappa}_{\tilde{k}+1}<(1+\mathcal{L}^*\Delta t)\hat{\kappa}_{\tilde{k}}
			<\cdots<(1+\mathcal{L}^*\Delta t)^{\tilde{k}}\hat{\kappa}_1\leq e^{\mathcal{L}^*T}\mathcal{K}^*.
		\end{align*}
		For $\tilde{k}=1,2,\cdots,N$, combining with \eqref{thm:eq14}, we can obtain
		\begin{align*}
			\Vert\nabla u_{h}^{\tilde{k}+1}\Vert^2<\hat{\kappa}_{\tilde{k}+1}<e^{\mathcal{L}^*T}\mathcal{K}^*,
		\end{align*}
		and
		\begin{align*}
			\Vert\nabla u_{h}^{1}\Vert^2<\mathcal{K}^*\leq e^{\mathcal{L}^*T}\mathcal{K}^*.
		\end{align*}
		Therefore, for $n=1,2,\cdots,N$,
		\begin{align}
			\Vert\nabla u^n_h\Vert^2< C_2(\epsilon, \sigma, u_0, c, T).
			\label{thm:eq15}
		\end{align}
		Substituting \eqref{thm:eq15} into \eqref{thm:eq0} yields
		\begin{align}
			\Vert u_h^n\Vert\leq C(\epsilon, \sigma, u_0, m, c, T,\vert\Omega \vert),~~n=1,2,\cdots,N.\label{addeq:4}
		\end{align}
		Since $\Phi^{'}_+(u^n_h)=(u^n_h)^3$ for $n=1,~2,\cdots,N$
		and by Sobolev's inequality (see Theorem 3 in \cite{P.Q.2012}),
		it follows
		\begin{align*}
			\Vert \Phi^{'}_{+}\Vert=\Vert(u^n_h)^3\Vert=\Vert u^n_h\Vert_{L^6(\Omega)}^3
			\leq c^3_S\Vert u^n_h\Vert^3_{H^1(\Omega)}.
		\end{align*}
		Now, we put $v_h=\Delta_h u^n_h$ in \eqref{eq:fullCH-2} and have
		\begin{align*}
			(w^n_h,\Delta_h u^n_h)&=\epsilon^2(\nabla u^n_h, \nabla \Delta_h u^n_h)
			+(\Phi'_+(u^n_h)-\Phi'_-(u^{n-1}_h), \Delta_h u^n_h)\\
			&=-\epsilon^2(\Delta_h u^n_h, \Delta_h u^n_h)+(\Phi'_+(u^n_h)-\Phi'_-(u^{n-1}_h), \Delta_h u^n_h),
		\end{align*}
		thus
		\begin{align*}
			\epsilon^2\Vert \Delta_h u^n_h\Vert^2
			&=-(w^n_h,\Delta_h u^n_h)+(\Phi'_+(u^n_h)-\Phi'_-(u^{n-1}_h), \Delta_h u^n_h)\\
			&\leq \Vert w^n_h\Vert\,\Vert\Delta_h u^n_h\Vert+\Vert\Phi'_+(u^n_h)\Vert\,\Vert\Delta_h u^n_h\Vert
			+\Vert\Phi'_-(u^{n-1}_h)\Vert\,\Vert\Delta_h u^n_h\Vert,
		\end{align*}
		\textit{i.e.,}
		\begin{align}
			\epsilon^2\Vert \Delta_h u^n_h\Vert
			&\leq \Vert w^n_h\Vert+\Vert\Phi'_+(u^n_h)\Vert +\Vert\Phi'_-(u^{n-1}_h)\Vert \notag\\
			&\leq \Vert w^n_h\Vert +c^3_S\Vert u^n_h\Vert^3_{H^1(\Omega)}+\Vert u^{n-1}_h\Vert.\label{addeq:2}
		\end{align}
		Moreover, we take $v_h=w^n_h$ in \eqref{eq:fullCH-2} and obtain
		\begin{align*}
			\Vert w^n_h\Vert^2
			&=\epsilon^2(\nabla u^n_h, \nabla w^n_h)+(\Phi'_+(u^n_h)-\Phi'_-(u^{n-1}_h), w^n_h)\\
			&\leq\epsilon^2\Vert \nabla u^n_h\Vert \,\Vert \nabla w^n_h\Vert
			+(\Vert \Phi'_+(u^n_h)\Vert+\Vert\Phi'_-(u^{n-1}_h)\Vert)\,\Vert w^n_h\Vert,
		\end{align*}
		then
		\begin{align}
			\Vert w^n_h\Vert
			&\leq\epsilon^2\Vert \nabla u^n_h\Vert+\Vert\Phi'_+(u^n_h)\Vert+\Vert\Phi'_-(u^{n-1}_h)\Vert.\label{addeq:3}
		\end{align}
		Hence, combining \eqref{addeq:3} with \eqref{addeq:2}, we have
		\begin{align*}
			\epsilon^2\Vert \Delta_h u^n_h\Vert
			&\leq \epsilon^2\Vert\nabla u^n_h\Vert+2c^3_S\Vert u^n_h\Vert^3_{H^1(\Omega)}+2\Vert u^{n-1}_h\Vert\\
			&\leq \epsilon^2\Vert \nabla u^n_h\Vert+2c^3_S(\Vert\nabla u^n_h\Vert^2+\Vert u^n_h\Vert^2)^{3/2}+2\Vert u^{n-1}_h\Vert.
		\end{align*}
		From \eqref{thm:eq15} and \eqref{addeq:4}, we obtain
		\begin{align}
			\Vert \Delta_h u^n_h\Vert \leq C(\epsilon, \sigma, u_0, m, c, c_s, T,\vert \Omega\vert),~~n=1,2,\cdots,N.
				\label{eq:proofinfty}
		\end{align}
		Assume that the finite element triangulation is quasi-uniform, it follows that
		\begin{align*}
			\Vert u^n_h-\bbint_{\Omega}u^n_h d\Omega\Vert_{L^{\infty}(\Omega)}
			=\Vert u^n_h-m\Vert_{L^{\infty}(\Omega)}
			\leq C\Vert u^n_h\Vert^{1-\theta}\Vert\Delta_h u^n_h\Vert^{\theta}.
		\end{align*}
From \eqref{addeq:4} and \eqref{eq:proofinfty}, we have
\begin{align*}
\Vert u^n_h-m\Vert_{L^{\infty}(\Omega)}\leq
C(\epsilon, \sigma, u_0, m, c,c_s, T,\vert \Omega\vert).
\end{align*}
Therefore, we can deduce that
		\begin{align*}
			\Vert u_h^n\Vert_{L^{\infty}(\Omega)}\leq C(\epsilon, \sigma, u_0, m, c,c_s, T,\vert \Omega\vert),~~n=1,2,\cdots,N,
		\end{align*}
Hence, the proof of Theorem \ref{thm:bound} is completed.
	\end{proof}
	
	\section{Proof of Theorem \ref{th:Newton Error} }\label{secA3}
	
	The following lemmas are helpful to Theorem \ref{th:Newton Error}.
	\begin{lemma}\label{lemma:1}
		\cite{X.1996} For any $ u_1,~w_1,~u_2,~w_2,~v\in X$, denote
		\begin{align}		
			\Psi(u_1,w_1;v)=\Psi(u_2,w_2;v)+\Psi'(u_2,w_2;u_1-u_2,w_1-w_2,v)
			+R(u_2,w_2,u_1,w_1,v).\label{le:1}
		\end{align}
		For any given $k_0>0$, if
		\begin{align}
			\Vert u_1\Vert_{1,\infty}\leq k_0,
			~~\Vert u_2\Vert_{1,\infty}\leq k_0,
			\label{le:0}
		\end{align}
		then the remainder $R$ satisfies
		\begin{align*}
			\vert R(u_2,w_2,u_1,w_1,v)\vert=\vert R_1(u_1,u_2,v)\vert \leq C(k_0)\Vert e\Vert_{0,2p}^2\cdot\Vert v\Vert_{0,q},
		\end{align*}
		where $C(k_0) \geq \max\limits_{x\in\Omega,\vert y\vert \leq k_0}\{\vert g_{1,yy}\vert \}$ and
		$e=u_1-u_2$, $\frac{1}{p}+\frac{1}{q}=1, p,~q\geq 1$.
	\end{lemma}
	
	\begin{lemma}
		\label{lemma:2}
		\cite{H.2013}	
		If $h$ is sufficiently small and $u\in X$ is an isolated solution of (\ref{eq:CH}), then
		\begin{align*}
			\vert \Psi^{'}_1(u;v,\phi)\vert &\leq c \Vert v\Vert_{1,p}\Vert \phi\Vert_{1,q},
			~~\forall~v\in W^{1,p}(\Omega),~\phi\in W^{1,q}(\Omega).
		\end{align*}
	\end{lemma}
	
	\begin{lemma}
		\label{lemma:4}
		\cite{R.1982}\cite{S.1974}	
		For a given $z\in\Omega$, we can find $\hat{g}_{h}^z\in V_h$ such that
		\begin{align*}
			\Psi^{'}_1(u; v_h, \hat{g}_{h}^z)=\partial v_h(z),~~v_h\in V_h,
		\end{align*}
		where
		\begin{align*}
			\sup_{z\in \bar{\Omega}}\Vert\tilde{g}_{h}^z(z)\Vert_{1,1}\leq c\, \vert \log h\vert.
		\end{align*}
	\end{lemma}
	
	\begin{proof}
		Using the definition (\ref{le:1}) in Lemma \ref{lemma:1}, we have
		\begin{align}
			\Psi_1(u^{n,k-1}_h,v_h)&+\Psi_2(w^{n,k-1}_h,v_h)
			+\Psi^{'}_1(u_h^{n,k-1};u_h^{n}-u_h^{n,k-1},v_h)\notag\\
			&~~~~+\Psi^{'}_2(w_h^{n,k-1};w_h^n-w_h^{n,k-1},v_h)
			+R_1(u_h^{n,k-1},u^{n}_h,v_h)=0.\label{Newton:3}
		\end{align}
		Combining (\ref{Newton:2}) with (\ref{Newton:3}), it is obvious that
		\begin{align}
			&\Psi_1(u^{n,k-1}_h,v_h)+\Psi_2(w^{n,k}_h,v_h)
			+\Psi^{'}_1(u_h^{n,k-1};u_h^{n,k}-u_h^{n,k-1},v_h)=0,\label{3.6.1}\\
			&\Psi_1(u^{n,k-1}_h,v_h)+\Psi_2(w^n_h,v_h)
			+\Psi^{'}_1(u_h^{n,k-1};u^{n}_h-u_h^{n,k-1},v_h)
			+R_1(u_h^{n,k-1},u^{n}_h,v_h)=0.\label{3.6.2}
		\end{align}
		Subtracting (\ref{3.6.2}) from (\ref{3.6.1}),
		we have the following error equation
		\begin{align}	
			\Psi_2(w^n_h-w_h^{n,k},v_h)
			+\Psi^{'}_1(u_h^{n,k-1};u^n_h-u_h^{n,k},v_h)
			+R_1(u_h^{n,k-1},u^n_h,v_h)=0.
			\label{Newton:04}
		\end{align}
		1) Firstly, using mathematical induction method for $\tilde{m}$, we prove
		the following estimate
		\begin{align}
			\Vert u^n_h-u_h^{n,\tilde{m}}\Vert_{1,\infty}
			\leq \rho^{2^{\tilde{m}}-1}\eta^{2^{\tilde{m}}}.\label{3.7}
		\end{align}
		
		i) When $\tilde{m}=0$, \eqref{3.7} is evidently true according to
		the initial condition \eqref{initial}.
		
		ii) Assume that \eqref{3.7} is true for $\tilde{m}=k-1$.
		
		iii) It is only required to prove \eqref{3.7} for $\tilde{m}=k$.
		From \cite{P.Q.2012} and \cite{H.2013}, we know that
		\begin{align*}
			\Vert u\Vert_{1,\infty}\leq C(\epsilon,\sigma,u_0,c_p,c_s,T,m,\vert \Omega\vert)
		\end{align*}
	and
\begin{align*}
\Vert u-u^n_h\Vert_{1,\infty}\leq ch= c_2.
	\end{align*}	
	are true. Thus, according to the induction assumption of $\tilde{m}=k-1$,
		we have
		\begin{align*}
			\Vert u-u_h^{n,k-1}\Vert_{1,\infty}&\leq
			\Vert u-u^n_h\Vert_{1,\infty}+\Vert u^n_h-u_h^{n,k-1}\Vert_{1,\infty}\\
			&\leq c_2+(\rho\eta)^{2^{n-1}-1}\eta\leq c_2+\eta.
		\end{align*}
		Then
		\begin{align}
			\Vert u_h^{n,k-1}\Vert_{1,\infty}\leq C+c_2+\eta\leq k_0.
			\label{3.7.1}
		\end{align}
		According to the definition of $\Psi^{'}_1(\cdot;\cdot,\cdot)$, we obtain
		\begin{align}	
			\vert \Psi^{'}_1(u;u^{n,k}_h-u^n_h,v_h)
			-\Psi^{'}_1(u^{n,k-1}_h;u^{n,k}_h-u^n_h,v_h)\vert 
			\leq \mathcal{F}(u,u^{n,k-1}_h)
			\Vert u^{n,k}_h-u^n_h\Vert_{1,\infty}\Vert v_h\Vert_{1,1},
			\label{le4:0}
		\end{align}
		where
		\begin{align}
			\mathcal{F}(u,u^{n,k-1}_h)
			=&\Vert F_{1,z}(x,u,\nabla u )-F_{1,z}(x,u^{n,k-1}_h,\nabla u^{n,k-1}_h)\Vert_{0,\infty}\notag\\
			&+\Vert F_{1,y}(x,u,\nabla u)-F_{1,y}
			(x,u^{n,k-1}_h,\nabla u^{n,k-1}_h )\Vert_{0,\infty}\notag\\
			&+\Vert g_{1,z}(x,u,\nabla u)
			-g_{1,z}(x,u^{n,k-1}_h,\nabla u^{n,k-1}_h )\Vert_{0,\infty}\notag\\
			&+\Vert g_{1,y}(x,u,\nabla u)
			-g_{1,y}(x,u^{n,k-1}_h,\nabla u^{n,k-1}_h )\Vert_{0,\infty}\notag\\
			=&\Vert g_{1,y}(x,u,\nabla u)
			-g_{1,y}(x,u^{n,k-1}_h,\nabla u^{n,k-1}_h )\Vert_{0,\infty}.
			\label{le4:2}
		\end{align}
		By setting
		\begin{align*}
			\Phi(t)=g_{1,y}(x,u+t(u^{n,k-1}_h-u),\nabla u_1
			+t\nabla (u^{n,k-1}_h-u)),
		\end{align*}
		for $u,~u^{n,k-1}_h\in X\cap W^{1,\infty}(\Omega)$ and using the definition of $g_{1,y}$, we have
		\begin{align}
			\vert g_{1,y}(x,u^{n,k-1}_h,\nabla u^{n,k-1}_h )-g_{1,y}(x,u,\nabla u )\vert 
			&=\Big{\vert} \int_0^1\Phi^{'}(t)\,dt \Big{\vert} \notag\\
			&\leq \max_{x\in\Omega,\vert y\vert \leq k_0,\vert z\vert \leq k_0}(\vert g_{1,yy}\vert)\Vert u-u^{n,k-1}_h\Vert_{1,\infty}.
			\label{le4:3}
		\end{align}
		Combining (\ref{le4:2}) with (\ref{le4:3}), it follows that
		\begin{align}
			\mathcal{F}(u,u^{n,k-1}_h)\leq C(k_0)\Vert u-u^{n,k-1}_h\Vert_{1,\infty}.
			\label{le4:4}
		\end{align}
		Substituting (\ref{le4:4}) into (\ref{le4:0}) yields
		\begin{align}	
			&\vert \Psi^{'}_1(u;u^{n,k}_h-u^n_h,v_h)
			-\Psi^{'}_1(u^{n,k-1}_h;u^{n,k}_h-u^n_h,v_h)\vert \notag\\
			&\leq C(k_0)\Vert u-u^{n,k-1}_h\Vert_{1,\infty}
			\Vert u^{n,k}_h-u^n_h\Vert_{1,\infty}\Vert v_h\Vert_{1,1}\notag\\
			&\leq (c_2+\eta)C(k_0)\,\Vert u^{n,k}_h-u^n_h\Vert_{1,\infty}\Vert v_h\Vert_{1,1}.
			\label{3.7.2}
		\end{align}
		Moreover, using Lemma \ref{lemma:2} and (\ref{3.7.1}), it is evident that
		\begin{align}
			\vert \Psi^{'}_1(u^{n,k-1}_h;u^{n,k}_h-u^n_h,v_h)\vert
			&\leq (c_2+\eta)\,\Vert u^{n,k}_h-u^n_h\Vert_{1,\infty}\, \Vert v_h\Vert_{1,1}+\vert \Psi^{'}_1(u;u^{n,k}_h-u^n_h,v_h)\vert \notag\\
			&\leq (c_2+\eta)\,\Vert u^{n,k}_h-u^n_h\Vert_{1,\infty}\,\Vert v_h\Vert_{1,1} + c_1 \Vert u^{n,k}_h-u^n_h\Vert_{1,\infty}\,\Vert v_h\Vert_{1,1}\notag\\
			&=[(c_2+\eta)+c_1]\,\Vert u^{n,k}_h-u^n_h\Vert_{1,\infty}\, \Vert v_h\Vert_{1,1}.\label{3.7.3}
		\end{align}
		Now for any $z\in\Omega$, let $\hat{g}^z_h$ be
		the discrete Green function.
		Note that
		\begin{align*}
			\partial\,(w^{n,k}_h-w^n_h)(z)
			= &\Psi^{'}_2(w;w^{n,k}_h-w^n_h,\hat{g}_h^z)\\
			=&\Psi^{'}_2(w;w^{n,k}_h-w^n_h,\hat{g}_h^z)
			+\Psi_2(w^n_h-w_h^{n,k},\hat{g}_h^z)\\	 &+\Psi^{'}_1(u_h^{n,k-1};u^n_h-u_h^{n,k},\hat{g}_h^z)
			+R_1(u_h^{n,k-1},u^n_h,\hat{g}_h^z)\\	 =&\Psi^{'}_1(u_h^{n,k-1};u^n_h-u_h^{n,k},\hat{g}_h^z)
			+R_1(u_h^{n,k-1},u^n_h,\hat{g}_h^z).
		\end{align*}
		Applying Lemma \ref{lemma:4} to \eqref{Newton:04}, we have
		\begin{align}
			\vert \partial\,(u^{n,k}_h-u^n_h)(z)\vert 
			=&\vert \Psi^{'}_1(u;u^{n,k}_h-u^n_h,\hat{g}_h^z)
			-\Psi_2(w^n_h-w_h^{n,k},\hat{g}_h^z)\notag\\	 &-\Psi^{'}_1(u_h^{n,k-1};u^n_h-u_h^{n,k},\hat{g}_h^z)
			-R_1(u_h^{n,k-1},u^n_h,\hat{g}_h^z)\vert \notag\\
			=&\vert \Psi^{'}_1(u;u^{n,k}_h-u^n_h,\hat{g}_h^z)
			-\partial\,(w^{n,k}_h-w^n_h)(z)\notag\\	 &-\Psi^{'}_1(u_h^{n,k-1};u^n_h-u_h^{n,k},\hat{g}_h^z)
			-R_1(u_h^{n,k-1},u^n_h,\hat{g}_h^z)\vert.
			\label{3.6.3}
		\end{align}
		Therefore, combining \eqref{3.7.2}, \eqref{3.7.3} with \eqref{3.6.3}, by Lemma \ref{lemma:1}, we obtain
		\begin{align}
			\vert \partial\,(u^{n,k}_h-u^n_h)(z)\vert 
			=&\vert \Psi^{'}_1(u;u^{n,k}_h-u^n_h,\hat{g}_h^z)
			-2\Psi^{'}_1(u_h^{n,k-1};u^n_h-u_h^{n,k},\hat{g}_h^z)
			-2R_1(u_h^{n,k-1},u^n_h,\hat{g}_h^z)\vert \notag\\
			\leq &\vert \Psi^{'}_1(u;u^{n,k}_h-u^n_h,\hat{g}_h^z)
			-\Psi^{'}_1(u_h^{n,k-1};u^n_h-u_h^{n,k},\hat{g}_h^z)\vert \notag\\	 &+\vert \Psi^{'}_1(u_h^{n,k-1};u^n_h-u_h^{n,k},\hat{g}_h^z)\vert 
			+\vert R_1(u_h^{n,k-1},u^n_h,\hat{g}_h^z)\vert \notag\\
			\leq& [(C(k_0)+1)(c_2+\eta)+c_1]\vert \log h\vert \,\Vert u^{n,k}_h-u^n_h\Vert_{1,\infty}\notag \\
			&+C(k_0)\vert \log h\vert \cdot \Vert u^n_h-u_h^{n,k-1}\Vert_{1,\infty}^2.\label{3.7.4}
		\end{align}
		Applying Lemma \ref{lemma:4} to \eqref{3.7.4} yields
		\begin{align*}
			\Vert u^{n,k}_h-u^n_h\Vert_{1,\infty}
			\leq & c_3[(C(k_0)+1)(c_2+\eta)+c_1]\vert \log h\vert \,\Vert u^{n,k}_h-u^n_h\Vert_{1,\infty}\\
			&+c_3C(k_0)\vert \log h\vert \,\Vert u^n_h-u_h^{n,k-1}\Vert^2_{1,\infty}.
		\end{align*}
		It can be seen that
		\begin{align}
			\Vert u^{n,k}_h-u^n_h\Vert_{1,\infty}
			&\leq\frac{c_3C(k_0)\vert \log h\vert }{1-c_3[(C(k_0)+1)(c_2+\eta)+c_1]\vert \log h\vert }\Vert u^n_h-u_h^{n,k-1}\Vert^2_{1,\infty}\notag \\
			&=\tilde\rho \Vert u^n_h-u_h^{n,k-1}\Vert^2_{1,\infty}.\label{3.7.5}
		\end{align}
		Using induction assumption, it is evident that
		\begin{align}
			\Vert u^n_h-u_h^{n,k-1}\Vert^2_{1,\infty}\leq \tilde\rho^{2\times (2^{k-1}-1)}\eta^{2\times2^{k-1}}.\label{3.10}
		\end{align}
		Combining \eqref{3.7.5} with \eqref{3.10} yields
		\begin{align*}
			\Vert u^n_h-u_h^{n,k}\Vert_{1,\infty}\leq \tilde\rho^{ 2^k-1}\eta^{2^k}=c\rho^{2^k},
		\end{align*}
	where $\rho =\tilde\rho\eta$. Therefore, (\ref{3.7}) is ture for any $k$.
		
		2) Secondly, note that
		\begin{align}
			\Vert u^n_h-u_h^{n,{\tilde{m}}}\Vert_{1,p}\leq \vert \Omega\vert^{\frac{1}{p}}\Vert u^n_h-u_h^{n,{\tilde{m}}}\Vert_{1,\infty}.
			\label{3.12}
		\end{align}
		In terms of the condition (\ref{3.13}), substituting (\ref{3.12}) into (\ref{3.7}) yields
		\begin{align*}
			\Vert u^n_h-u_h^{n,{\tilde{m}}}\Vert_{1,p}\leq c\rho^{2^{\tilde{m}}}.
		\end{align*}
		The proof is completed.
	\end{proof}

	
	
	
\end{appendices}

\end{document}